\documentclass{article}

\usepackage[utf8]{inputenc}
\usepackage{lmodern}
\usepackage{textgreek}

\usepackage{amsmath}
\usepackage{amsthm}
\usepackage{mathtools}
\usepackage{amsfonts}
\usepackage{wasysym}
\usepackage{bbm}
\usepackage{thmtools} %new math ambient - also for cross references
\usepackage{stackengine,graphicx,trimclip}
\usepackage{xargs}
\usepackage[hidelinks]{hyperref} %with hidden color
\usepackage[capitalize, nameinlink]{cleveref}

\newcounter{relctr} %% <- counter for relations
\newcommand\labelrel[2]{%
  \begingroup
    \refstepcounter{relctr}%
    \stackrel{\textnormal{(\theequation.\arabic{relctr})}}{\mathstrut{#1}}%
    \originallabel{#2}%
  \endgroup
}
\AtBeginDocument{\let\originallabel\label} %% <- store original definition

\usepackage{graphicx}
\usepackage{xcolor}

\usepackage[margin=3.4cm]{geometry}
 
\definecolor{Bnavy}{RGB}{0, 66, 128}
\definecolor{Bdust}{RGB}{140,179,217}
\definecolor{Bsugarpaper}{RGB}{198, 217, 236}
\definecolor{guppiegreen}{rgb}{0.0, 1.0, 0.5}
\definecolor{frenchrose}{rgb}{0.96, 0.29, 0.54}
\usepackage{makecell}

\usepackage{microtype}
\usepackage{adjustbox}
\usepackage{tikz} 
\usepackage{tikz-cd}
\usetikzlibrary{ %snake arrows
	calc,%
	arrows,%
	shapes,
	shapes.geometric,
    positioning,
    fit,
} 

\usepackage{esvect}
\usepackage{stackengine,scalerel}

\tikzcdset{arrow style=tikz, diagrams={>=stealth}}

\usepackage[labelsep=quad,indention=10pt]{subfig}
\theoremstyle{plain}
\newtheorem{theorem}{Theorem}[section]

\newtheorem{lemma}[theorem]{Lemma}
\newtheorem{corollary}[theorem]{Corollary}
\newtheorem{proposition}[theorem]{Proposition}
\newtheorem{remark}[theorem]{Remark}
\newtheorem{notation}[theorem]{Notation}

\theoremstyle{definition}
\newtheorem{definition}[theorem]{Definition}
\newtheorem{example}[theorem]{Example}

\newcommand{\Noise}{\mathrm{Ns}}
\newcommand{\Amp}{\mathrm{Amp}}
\newcommand{\Pseudom}{\mathrm{Pseudom}}
\newcommand{\Vect}{\mathrm{Vect}}

\newcommand{\A}{\mathcal{A}}

\newcommand{\Pos}{\mathcal{Q}}

\newcommand{\N}{\mathbb{N}}

\newcommand{\ch}{\mathrm{Ch}}
\newcommand{\id}{\mathrm{id}}
\newcommand{\AlgInt}{\mathfrak{X}^{\textnormal{int}}}
\newcommand{\Hom}{\textnormal{Hom}}
\newcommand{\define}[1]{{\bf \boldmath{#1}}}
\stackMath
\newcommand{\frightarrow}{\scalebox{.4}[.3]{%
  \clipbox{0pt 0pt 2pt 0pt}{$\blacktriangleright$}\kern1.8pt}}
\newcommand{\fleftarrow}{\scalebox{.4}[.3]{%
  \kern1.8pt\clipbox{2pt 0pt 0pt 0pt}{$\blacktriangleleft$}}}
\newcommand{\subsetr}{\mathrel{%
  \stackengine{-1.1pt}{\subseteq}{\frightarrow\kern.3pt}{U}{r}{F}{T}{S}}}
\newcommand{\subsetl}{\mathrel{%
  \stackengine{-1.1pt}{\subseteq}{\kern.4pt\fleftarrow}{U}{l}{F}{T}{S}}}
\newcommand{\dirdiam}{\vv{\mu}}

\newcommand\obullet[1]{\ThisStyle{\ensurestackMath{%
  \stackon[1pt]{\SavedStyle#1}{\SavedStyle\kern.6\LMpt\bullet}}}}
\newcommand{\cont}{\overset{\circ}{\mu}}
\DeclareMathOperator{\coker}{coker}

\DeclareMathOperator{\ob}{ob}

\newcommand{\LI}[1]{#1\mathrm{LI}}
\newcommand{\NA}{\mathrm{NA}}
\newcommand{\AN}{\mathrm{AN}}

\newcommand{\AmpMeas}[1][\alpha]{\mu_{#1}}

\newcommand{\dist}{\mathrm{d}}
 
\newcommand{\shift}{\text{\textvarsigma}}

\newcommand{\Field}{\mathbb{F}}

\newcommand{\hilb}[1]{\textnormal{Hilb}\left(#1\right)}
\newcommand{\hilbp}[2]{\textnormal{Hilb}^{\, #2}\left(#1\right)}

\newcommandx{\permodx}[2][1= \Pos, 2=\mathfrak X]{\mathrm{PerM}_{#2}(#1)}
\newcommand{\permod}[1][n]{\mathrm{PerM}_{\mathfrak C}(\mathbb{R}^{#1})}
\newcommand{\permodq}[1]{\mathrm{PerM}(#1)}

\newcommand{\pnorm}[2][p]{||#2||_{#1}}
\newcommand{\eps}{\varepsilon}
\newcommand{\posr}{\left[0,\infty\right)}
\newcommand{\posrc}{\left[0,\infty\right]}
\newcommand{\ns}{\mathcal{S}}

\newcommand{\cCubes}{\underline{\mathfrak C}}
\usepackage{extpfeil} %two heads arrow with symbols
\usepackage[all]{xy}

\title{Amplitudes in persistence theory}

\author{Barbara Giunti\footnote{Department of Mathematics and Statistics, SUNY - University at Albany, Albany, NY, USA, and Institute of Geometry, Graz University of Technology, Graz, Austria},
John S.\ Nolan\footnote{Department of Mathematics, UC Berkeley, Berkeley, CA, US},
Nina Otter\footnote{Department of Mathematics, UCLA, Los Angeles, CA, US, and School of Mathematical Sciences, Queen Mary University of London, London, Unitek Kingdom}, Lukas Waas\footnote{Department of Mathematics, Heidelberg University, Germany}
}
\date{}

\begin{document}

\maketitle

\begin{abstract}
The use of persistent homology in applications is justified by the validity of certain stability results. 
At the core of such results is a notion of distance between the invariants that one associates with data sets.  
Here we introduce a general framework to compare distances and invariants in multiparameter persistence, where there is no natural choice of invariants and distances between them. 
We define amplitudes, monotone, and subadditive invariants that arise from assigning a non-negative real number to objects of an abelian category. 
We then present different ways to associate distances to such invariants, and we provide a classification of classes of amplitudes relevant to topological data analysis. 
In addition, we study the the relationships as well as the discriminitative power of such amplitude distances arising in topological data analysis scenarios.
\end{abstract}

\section{Introduction}

One-parameter persistent homology (PH) is, arguably, the most successful method in Topological Data Analysis (TDA), since it has been used in a variety of applications, ranging from nanoporous materials \cite{lee_quantifying_2017} to pulmonary diseases \cite{belchi_lung_2018}, with new preprints being released on a weekly basis\footnote{A database of applications of methods in TDA, including persistent homology, is being maintained at \cite{database}, and at the time of writing it includes more than two hundred applications of persistent homology. The entries in the database can be filtered by application area, data format, and topological methods.}. 

In one-parameter PH, one associates a $1$-parameter family of spaces to a data set, and by computing the homology of this family one then extracts algebraic invariants, called ``barcodes'', which give a summary of how topological invariants, such as the number of connected components, holes, and voids, evolve as the parameter value changes. 
PH can meaningfully be used for data analysis thanks to stability results, which, roughly, guarantee that the pipeline from the input data to the barcodes is Lipschitz, with respect to appropriate notions of distances on the input data and barcodes. 
Such distances are well-understood from a theoretical point of view since they rely on matchings between intervals in barcodes, and they differ in how the cost between matchings is computed.
In what is called the ``bottleneck distance'', one considers the cost of matching the largest intervals, while in what are generally called ``$p$-Wasserstein distances'' one takes the weighted sum of the costs of the matching over all the intervals. 
Of these distances, the bottleneck distance has been most widely used in applications \cite{atom-specific_2020, hofer_deep_2017,rohrscheidt_topological_2023} and has been studied the most from a theoretical point of view. 
However, for many applications, the bottleneck distance is too coarse, and there have recently been renewed efforts in studying stability results for $p$-Wasserstein distances \cite{skraba_turner}.

In many applications, one is interested in studying data sets across multiple parameters, and thus in what is called ``multiparameter persistent homology'' one wishes to study the homology of a multiparameter family of spaces associated with a data set. 
It is well-known that there is no generalization of the barcode for multiparameter families of spaces \cite{CZ09}. 
In practice, one thus chooses invariants that are meaningful for a specific problem, but which do not give a complete characterization of the resulting persistence module.  

Choosing an appropriate invariant in multiparameter persistence requires defining what it means to be ``persistent'', so as to be able to distinguish between topological signal and noise. 
To do this, one could first focus either on finding an appropriate notion of signal, or of noise.
Several different approaches, focused on finding an appropriate notion of signal, have been taken in the literature. 
The first invariant that was suggested is the ``rank invariant'' \cite{CZ09}, which encodes the information on the ranks of structure morphisms in the module.
More recently, new invariants have been proposed, thanks to the introduction of new techniques into the field, stemming from representation theory \cite{buchet2018realizations,BotnanLO20,BC20}, and commutative algebra \cite{HOST,T19,miller2020essential, BOO21}. 
To be able to study the stability of such invariants, one usually needs to first associate suitable distances with them. 
Studying distances associated with such choices of invariants is currently one of the most active areas of research in TDA. Several different notions have been proposed, including
the interleaving distance which, in an appropriate sense, generalizes the bottleneck distance for multiparameter persistence, see \cite{Bubenik2014,isometry1,Lesnick2015}. 
More recently, generalizations of Wasserstein-type distances and $\ell^{p}$-distances have been studied in \cite{BSS18} and \cite{bjerkevik2021ellpdistances}, respectively.

In a complementary line of work, the authors of \cite{SC+19} proposed an approach to study notions of ``noise'' for multiparameter persistence modules, with the idea of stabilizing invariants by minimizing them on disks, rather than proving a stability result for each of them.
Such an approach to minimizing invariants was also taken, in the $1$-parameter setting, in \cite{Bauer2017}.

In our work, we take a different perspective: we focus on classes of invariants that capture a certain notion of ``size'' or ``measure'' of persistence and on ways of associating distances with such invariants.
Specifically, given an abelian category $\A$, we are interested in studying real-valued maps $\alpha \colon \ob \A\to [0,\infty]$ that mimic the properties that we would expect a size function to satisfy. 
Namely, we would like the map to be monotone with respect to subobjects and quotients, and to be subadditive with respect to short exact sequences (see Section \ref{S:amplitudes}). 
We call such functions \textit{amplitudes}. 
One central property of amplitudes is that one may associate to each such measure of ``size'' an associated path-distance, which roughly measures the cost of transforming one object into another in terms of the amplitude. Compared to the general cost functions of \cite{BSS18}, the arising path distances have the technical advantage that one only needs to consider spans, instead of general paths to compute the distance.
It turns out that a similar observation in a complementary setting was already made in the framework of noise system in \cite{SC+19,oliver}. 
In fact, as we demonstrate in \cref{S:amplitudes}, the theories of amplitudes and noise systems and their associated metrics are essentially equivalent (see specifically \cref{prop:equ_of_cat,cor_commutative diagram}).  

In \cref{S:stability} we explain how the most prominent distances used in topological data analysis (including Bottleneck and Wasserstein distances \cite{Bubenik2014,isometry1,Lesnick2015,oudot,skraba_turner}) all arise from the distance associated to an amplitude (up to equivalence of metrics). 
We give an overview of such distances and illustrate how this perspective can be used to expose relationships between distances in terms of the relationships of their amplitudes.

Since our motivation stems from studying invariants arising from $n$-parameter persistence modules, we then proceed to investigate the class of amplitudes defined on categories of multi-parameter modules.
To an amplitude $\alpha$ defined for $n$-parameter persistence modules, we associate a function $\mu_{\alpha}$ defined on intervals in $\mathbb R^n$, by evaluating $\alpha$ on interval modules. 
We then study $\alpha$ in terms of this function $\mu_{\alpha}$.
In \cref{S_classification_amplitudes} we study two classes of amplitudes which turn out to be determined entirely in terms of their functions on intervals, called \textit{quasi-simple} and \textit{additive} amplitudes, respectively. 
The former is characterized by the property of behaving like a maximum under direct sums of persistence modules. 
Their associated function on intervals behaves like computing the diameter of a set in the direction of the positive cone. We call such functions on intervals \textit{directed diameters} (see \cref{def_directeddiam}). 
We then prove that passing to functions on intervals induces a bijection between quasi-simple amplitudes on sufficiently constructible multiparameter persistence modules and directed diameters on the set of (constructible) intervals of the underlying poset (\cref{thm:classification_of_quasi_simple_amp}), thus providing a complete classification of quasi-simple amplitudes. 
This provides an extension of a classification result of \cite{oliver}.
In similar fashion, we prove a classification result for additive amplitudes, i.e. amplitudes that behave additively under short exact sequences. They turn out to be in one-to-one correspondence with respect to functions on intervals which (amongst other properties) are additive under disjoint union, so-called \textit{contents} (\cref{thm:classification_of_additive_amp}). The classification result may be read as: Every additive amplitude on (sufficiently constructible) persistence modules is given by integrating the Hilbert function with respect to a specific choice of content. 

We then proceed to study distances associated with amplitudes on categories of persistence modules. 
In general, such distances are extended pseudometrics on the class of isomorphism classes.
In other words, it is generally not clear that two persistence modules of distance $0$ are isomorphic. 
This holds for the interleaving distance and finitely presented modules\cite{Lesnick2015}. 
In a similar spirit, a result for convolution distances of sheaves was recently shown in \cite{petit2021property}.
In our work, we show that in the general framework of amplitudes on sufficiently constructible persistence modules the associated pseudometric is a metric exactly when it is induced by an amplitude whose value is $0$ only for zero objects (see Section \ref{S:discriminativity}).

\section{Amplitudes on abelian categories} \label{S:amplitudes}

We begin with the definition of amplitudes, which are real-valued maps that behave well with respect to the structure of abelian categories.

\begin{definition}\label{D:amplitude}
	For an abelian category $\A$, a class function ${\alpha\colon \ob\A \to [0,\infty]}$ is called an \define{amplitude} if $\alpha(0) = 0$ and for every short exact sequence
	\begin{align*}
	0 \to A \to B \to C \to 0
	\end{align*} 
	in $\A$ the following inequalities hold:
	\begin{alignat}{2}\label{eq_monotonicity}
	& \text{(Monotonicity)} \qquad  &&
    \begin{array}{l}
	\alpha(A) \leq \alpha(B) \\
	\alpha(C) \leq \alpha(B)
	\end{array}
	\, ,
	\\
	& \text{(Subadditivity)}\qquad &&
    \begin{array}{l}\label{eq_subadditivity}
	\alpha(B) \leq \alpha(A) + \alpha(C) 
	\end{array}
	\, .
	\end{alignat}
	Moreover, if (\ref{eq_subadditivity}) is an equality, the amplitude is said to be \define{additive}, and if $\alpha(A)<\infty$ for all $A$ in $\ob\A$, then the amplitude is said to be \define{finite}.
\end{definition}  

We note that, as a direct consequence of the monotonicity property, an amplitude is an isomorphism invariant. 

\begin{example}\label{example:non-addtive-amplitude-ch}
(Max degree)
Consider the category $\ch$ of non-negatively-graded chain complexes over an abelian category, and chain maps between them. 
Define an amplitude $\delta$ on $\ch$ as
\[
\delta(C)\coloneqq \max \{n\in\mathbb{N} \ \vert \ C_{n}\neq 0\}
\]
for all $C$ in $\ob\ch$.
It is clear that $\delta(0)=0$.
Since a short sequence in $\ch$ is exact if and only if it is exact degreewise, we also have the properties of monotonicity (\ref{eq_monotonicity}) and subadditivity (\ref{eq_subadditivity}).
Moreover, $\delta$ is finite.
However, it is immediate to see that, in general, (\ref{eq_subadditivity}) is a strict inequality for this amplitude, and thus $\delta$ is not additive.
\end{example}

\begin{example}\label{example_p-norms}
In \cite{skraba_turner}, the authors consider $p$-norms of finitely presented 1-parameter persistence modules and use them to define distances isometric to the $p$-Wasserstein distances. 
Here, we note that $p$-norms are amplitudes.
Let $\oplus_{j\in J} \{\mathbb{I}\langle b_{j},d_{j}\rangle\}_{j\in J}$ be the finite barcode decomposition of a persistence module $M$ (see, for example, \cite{oudot} for the precise definition of these terms, which we will introduce in \cref{S_permod_finiteness}). 
The angular brackets denote that the interval can be open or closed on each endpoint.
Then the $p$-norm of $M$, for $p\in[1,\infty]$, is
\[
\rho_{p}(M)\coloneqq 
\begin{cases}
	\left(\displaystyle\sum_{j\in J} \vert d_{j}-b_{j}\vert^{p}\right)^{\frac{1}{p}} & \text{ if $p<\infty$}\\
	\max_{j\in J} \vert d_{j}-b_{j}\vert & \text{ if $p=\infty$}
\end{cases}
\, .
\]
Using the results of \cite{BauerLesnick14}, the authors of \cite{skraba_turner} show that such $p$-norms satisfy the monotonicity and subadditivity conditions (\ref{eq_monotonicity})-(\ref{eq_subadditivity}).
In particular, $\rho_1$ is an additive amplitude. 
\end{example}

\subsection{Span metrics} \label{S_path_span_metrics}

We now introduce metrics associated with amplitudes. 
A classical method to define a metric is to minimize the cost of a path between two objects.
By properties of amplitudes, the path can be shortened into a span, following the technique of \cite[Prop. 8.2 and Cor. 8.3]{SC+19}. 

\begin{definition}\label{def_cost_function}
A \define{cost function} is a function $f\colon \posrc^4\to\posrc$ satisfying the following conditions:
\begin{itemize}
    \item Monotonicity: $f(v) \leq f(w)$ whenever $v \leq w$.
    \item Subadditivity: $f(v + w) \leq f(v) + f(w)$ for any $v, w$.
    \item Behavior at zero: $f(0, 0, 0, 0) = 0$.
    \item Exchange law: $f(x_1,x_2,x_3,x_4)=f(x_3,x_2,x_1,x_4)=f(x_1,x_4,x_3,x_2)$.
\end{itemize}
\end{definition}

\begin{example}
The $1$-norm $\pnorm[1]{-}\colon \posrc^4 \to \posrc$, taking $\left(x_1,x_2,x_3,x_4\right)$ to $\sum_{i=1}^4 x_i$, is a cost function.
\end{example}

\begin{definition}\label{def_fcost}
Let $\A$ be an abelian category with amplitude $\alpha$.
Given a span ${\bf c}\colon A\xleftarrow{\ \varphi\ } C \xrightarrow{\ \psi\ } B$ in $\A$ and a cost function $f$, the \define{$f$-cost of c} is 
\[f_\alpha({\bf c})\coloneqq f(\alpha(\ker\varphi),\alpha(\coker\varphi),\alpha(\ker\psi),\alpha(\coker\psi))\, .
\]
One can similarly define the  $f$-cost of a cospan.
\end{definition}

\begin{example}
Let $\varphi\colon A \to B$ be a morphism in an abelian category $\A$ with amplitude $\alpha$. 
We can use \cref{def_fcost} to assign a cost to it:
\[
c_\alpha(\varphi)\coloneqq ||-||_{1\, \alpha}\left(B \xleftarrow{\varphi} A \rightarrow 0\right)=\alpha(\ker\varphi)+\alpha(\coker\varphi) \, .
\]
\end{example}

\begin{remark}\label{R_span_cospan}
For every cost function $f$, we have that, given a span $\mathbf{c}\colon A\xleftarrow{\ \varphi\ } C \xrightarrow{\ \psi\ } B$, there exists a cospan $\mathbf{c}'\colon A\xrightarrow{\ \psi'\ } C \xleftarrow{\ \varphi' \ } B$ with $f_\alpha(\mathbf{c}') \leq f_\alpha(\mathbf{c})$, and similarly with the roles of spans and cospans swapped. 
The proof is in \cite[Cor. 8.3]{SC+19} for $f= \pnorm[1]{-}$, but generalizes immediately to our setting.
\end{remark}

We use the $f$-cost to define a (pseudo, extended) metric:

\begin{definition}\label{def_span_metric}
Given an amplitude $\alpha$ on an abelian category $\A$ and a cost function, the \define{span metric associated to $f_\alpha$} is:
\[
\dist_{f_\alpha}\left(A,B\right)=
\inf\Big\{f_\alpha({\bf c}) \mid {\bf c} \colon A\xleftarrow{\ \varphi\ } C \xrightarrow{\ \psi\ } B \Big\}\, ,
\]
with the convention that $\inf\emptyset = \infty$.
When $f = \pnorm[1]{-}$, we simply say \define{span metric associated to $\alpha$}, denoted by $\dist_{\alpha}$.
\end{definition}

\begin{proposition}\label{P_f_span_metric}
Let $\A$ be an abelian category with amplitude $\alpha$, and let $f$ be a cost function.
Then the assignment $d_{f_\alpha}$ is a pseudometric on $\ob \A$.
\end{proposition}

\begin{proof}
It is clear that $\dist_{f_\alpha}(A,A)=0$ and that $\dist_{f_\alpha}(A,B)=\dist_{f_\alpha}(B,A)$ for every $A,B$ in $\ob\A$. 
For the triangle inequality, consider two spans ${\bf c_1}\colon A\xleftarrow{\ \varphi_1\ } H_1 \xrightarrow{\ \psi_1\ } B$ and  ${\bf c_2}\colon B\xleftarrow{\ \varphi_2\ } H_2 \xrightarrow{\ \psi_2\ } C$. 
By Remark \ref{R_span_cospan} applied to the cospan $H_1 \xrightarrow{\ \psi_1\ } B \xleftarrow{\ \varphi_2\ } H_2$, we obtain a span 
\[
{\bf c_3}\colon A\xleftarrow{\ \varphi_1\ } H_1 \xleftarrow{\ \varphi_3\ }H_3\xrightarrow{\ \psi_3\ } H_2 \xrightarrow{\ \psi_2\ } C.
\]
Thus, it is left to show that $f_\alpha({\bf c_3})\leq f_\alpha({\bf c_1})+f_\alpha({\bf c_2})$. 
This follows by the argument presented in \cite[Prop. 8.2]{SC+19} and the properties of $f$.
\end{proof}

\begin{remark}
This result is not true in the more general case of weight functions \cite{BSS18} and therefore motivates the use of amplitudes or noise systems.
Indeed, one of the motivations for the introduction of the axioms of noise systems \cite{SC+19} was to allow for studying distances using only (co)spans instead of arbitrary zigzags.
\end{remark}
\begin{example}\label{e:path_metrics_intervals}
An interval module is an object in $\text{Fun}\left(\mathbb{R}, \Vect\right)$ that assumes value $\Field$ over an interval of $\mathbb{R}$ of the form $[a,b)$, with the identity as a structure map, and $0$ otherwise.
Consider two interval modules $\mathbb{I}[a, a + r)$ and $\mathbb{I}[b, b + r)$, for $r>0\in\mathbb{R}$ and $a,b\in\mathbb{R}$ with $a+r<b$, and $\alpha = \rho_p$. 
We have $\dist_{\rho_{p}}(\mathbb{I}[a, a + r),\mathbb{I}[b, b+r))=2r$, and taking $f=\max$, $\dist_{f_{\rho_{p}}}(\mathbb{I}[a, a + r),\mathbb{I}[b, b+r))=r$.
\end{example}

\begin{lemma}\label{lemma:amplitude_Lipschitz}
Let $\A$ be an abelian category with amplitude $\alpha$. 
Then, for any $A,B$ in $\ob\A$, the following inequality holds:
\[
|\alpha(A) - \alpha(B)| \leq \dist_{\alpha}(A,B).
\] 
\end{lemma}

In other words, $\alpha$ is $1$-Lipschitz with respect to $\dist_{\alpha}$. 

\begin{proof}
By the triangle inequality and by definition of $\dist_\alpha$, it suffices to show that whenever there is a morphism $\varphi\colon A \to B$ with $c_{\alpha}(\varphi) \leq \varepsilon$, then 
\[ 
|\alpha(A) - \alpha(B)| \leq \varepsilon \, .
\]
Again by the triangle inequality and epi-mono factorization, we may restrict to the case where $\varphi$ is either an epi- or monomorphism. 
If $\varphi$ is a mono, we then have a short exact sequence $ 0 \to A \to B \to \coker{\varphi} \to 0$ with $\alpha(\coker(\varphi)) \leq \varepsilon$. Hence, by monotonicity and subadditivity of amplitudes follows that,
\[ 
\alpha(A) \leq \alpha(B) \leq \alpha(A) + \varepsilon \, ,
\] 
showing the required inequality. 
The case of $\varphi$ an epimorphism follows dually. 
\end{proof}

\begin{lemma}\label{L:char ampl metr}
Let $\A$ be an abelian category and $\alpha$ an amplitude on $\A$. 
Then, for each $A$ in $\ob\A$,
\[
\alpha(A) = \dist_{\alpha}(A,0) \, .
\] 
\end{lemma}

\begin{proof}
We observe that $\alpha(A) = c_{\alpha}(A\to 0)$, and thus $\dist_{\alpha}(A,0) \leq \alpha(A)$. 

For the converse inequality, consider a cospan $A\overset{\varphi}{\longrightarrow} B \longleftarrow 0$.
It suffices to show that $c_\alpha(A \to 0) \leq c_\alpha(\varphi) + c_\alpha(0 \to B)$ to obtain the missing inequality by passing to the infimum. 
Since $0 = 0 \circ \varphi $, we have 
\begin{equation*}
c_\alpha(A \to 0) \leq c_\alpha(\varphi) + c_\alpha(B \to 0) = c_\alpha(\varphi) + c_\alpha(0 \to B)\, .  \qedhere
\end{equation*}
\end{proof}

In data analysis, stability is a crucial property as it guarantees the robustness of the analysis: the extracted information (aka, the amplitude) does not vary excessively if the input data is only slightly perturbed. 
In more detail, we want to have that the span metric associated with a certain cost function of an amplitude is bounded by some (meaningful) metric on the input. 
This property is in general difficult to prove. 
To ease the process, we now present an elementary fact relating the stability of exact functors to span metrics. 
Using this technique, one can then prove a stability result for a certain span metric and then transfer it to other span metrics.

\begin{proposition}\label{prop:exact_ampl_con}
Let $g\colon \posrc \to \posrc$ be a monotonous continuous function.
Let $F\colon (\A,\alpha)\to (\A',\alpha')$ be an exact functor that satisfies $\alpha'(F(A))\leq g(\alpha(A))$ for all $A\in \ob\A$.
Then we have
\[ 
\dist_{\alpha'}(F(A),F(B))\leq 4 g(\dist_{\alpha}(A,B)) 
\]
for all $A,B\in \ob\A$. 
In particular,  $F$ induces a continuous function $\ob \A \to \ob \A'$ with respect to the topologies induced by $\dist_{\alpha}$ and $\dist_{\alpha'}$.
\end{proposition}
There is also the following immediate fact concerning changes in cost-processing function.

\begin{lemma}\label{composition_leq4}
Let $\A$ be an abelian category with amplitude $\alpha$, $f,h$ cost functions, and $g$ a monotonous function, continuous from above, such that $h\leq g \circ f$.
Then
\[
\dist_{h_\alpha}(A, B) \leq g(\dist_{f_\alpha}(A, B))
\]
for all $A, B \in \ob \A$. 
\end{lemma}

\begin{proof}[Proof of \cref{prop:exact_ampl_con}]
Let $f = \pnorm[1]{-}$.
Consider a (co)span ${{\bf c}\colon A\xleftrightarrow{\ \varphi\ } C \xleftrightarrow{\ \psi\ } B}$. 
As above, write $\alpha_1 = \alpha(\ker\varphi)$, $\alpha_2 = \alpha(\coker\varphi)$, $\alpha_3 = \alpha(\ker\psi)$, and $\alpha_4 = \alpha(\coker\psi)$, and proceed analogously for $F(\varphi), F(\psi)$, using $\alpha'_i$.
In particular, using the exactness of $F$, we have $\alpha'_i \leq g(\alpha_i)$.
It follows that
\begin{align*}
    \dist_{f_\alpha}(F(A), F(B)) & \leq f_{\alpha'}(F({\bf c})) \leq \sum_{i=1}^4 g(\alpha_i) \leq 4 \max_i g(\alpha_i) = 4 g\left(\max_i \alpha_i\right) 
    \\
    & \leq 4 g\left(\sum_{i=1}^4 \alpha_i \right) = 4 g( f_\alpha({\bf c})) \, .
\end{align*}
Now we take the infimum over all ${\bf c}$ and, using the fact that $4g$ commutes with infima, by the continuity assumption, we conclude the proof.
\end{proof}

\subsection{Noise systems and amplitudes}\label{S:conn noise}

The idea of amplitudes is closely related to that of a noise system, introduced in \cite{SC+19}. 
While the concept of amplitude arises from the aim to capture a notion of size for objects in an abelian category, noise systems were introduced to define stable approximations of persistence modules, and thus a noise system can be seen as a way to hierarchically organize persistence modules according to how small they are. 
Here we rephrase the definition of noise systems, introduced in \cite{SC+19} for tame persistence modules $\mathbb{Q}^n\to \Vect$, for arbitrary abelian categories and show that they are related to amplitudes by an equivalence between appropriately defined categories.

\begin{definition}\label{def_noise_system} 
Let $\A$ be an abelian category and $\posr$ the poset of non-negative reals.
A \define{noise system} in $\A$ is a collection $\ns=\{\ns_{\eps}\}_{\eps\in \posr}$ of collections of objects of $\A$, indexed by non-negative reals $\eps$, such that:
\begin{itemize}
\item[(i)] the zero object belongs to $\ns_{\eps}$ for any $\eps$;
\item[(ii)] if $0\leq \tau<\eps$, then $\ns_{\tau}\subseteq \ns_{\eps}$;
\item[(iii)] if $0\to F\to G\to H\to 0$ is an exact sequence in $\A$, then
\begin{itemize}
\item if $G$ is in $\ns_{\eps}$, then so are $F$ and $H$ (subobject property);
\item if $F$ is in $\ns_{\eps}$ and $H$ is in $\ns_{\tau}$, then $G$ is in $\ns_{\eps+\tau}$ (additivity property).
\end{itemize}
\end{itemize}
\end{definition}

In \cite{SC+19}, the authors provide several examples of noise systems. 
Among these examples, they discuss the property for a noise system of being closed under direct sum. 
Given a noise system $\ns$ on an abelian category $\A$, $\ns$ is closed under direct sum if, for all $\eps \in \mathbb{R}$ and for any pair of objects $M$, $N$ in $\ns_{\eps}$, then $M\oplus N$ is also in $\ns_{\eps}$. 
Later works \cite{oliver,henri} describe a machinery called \define{contour} that generates noise systems that are in particular closed under direct sums.
Since every noise system can be seen as an amplitude (see \cref{prop_noise_to_weight}), the theory described in \cite{oliver, henri} can be used to generate amplitudes that on direct sums take as value the maximum of the evaluations on single elements. We call such amplitudes quasi-simple, and classify them in the applied topology frameworks in \cref{SS:lower bound}.

\begin{definition}\label{def:shift_amplitude}
Let $C$ be a cone in $\mathbb{R}^{n}$, i.e.\ the set of all linear combinations with non-negative coefficients of a finite collection of non-negative elements in $\mathbb{R}^{n}$, and $||-||$ a norm on $\mathbb{R}^{n}$. 
Then the shift amplitude $\shift$ on the functor category $\text{Fun}\left(\mathbb{R}^n,\Vect\right)$ is defined as follows: 
for any $M$ in $\ob\text{Fun}\left(\mathbb{R}^n,\Vect\right)$,    
\[
    \shift(M)\coloneqq \inf\{ \ \pnorm[]{c} \ \vert \ c\in C \text{ and } M(t\leq t+ c) = 0 \text{ for all } t\in\mathbb{R}^{n}\}\, .
\]
The fact that this indeed is an amplitude follows from the properties of a noise system and the definition of standard noise in the direction of a cone in \cite[Par. 6.2]{SC+19}.

If the cone is generated by one element, then we get a simpler description of $\shift$, which also motivates the name.
Let $c\in C$ be the only vector in $C$ with $||c||=1$.
Then, for any $M$ in $\ob\text{Fun}\left(\mathbb{R}^n,\Vect\right)$, 
\[
    \shift(M)\coloneqq \inf\{\eps \in \mathbb{R} \ \vert \ M(t\leq t+ \eps c) = 0\} \, .
\]
In what follows, we often refer to the shift amplitude along a cone generated by one element as \define{shift along a vector} $v$.
Note that the shift amplitude $\shift$ does not need to be finite, in general. 
To see this it is enough to consider a module that persists forever in the direction of the cone. 
\end{definition}

We now study in detail the relation between amplitudes and noise systems. 
In particular, we will show that noise systems and amplitudes are the same by way of an adjunction.

\begin{definition} \label{def_noise_cat}
Let $(\A, \ns)$ and $ (\A', \ns')$ be abelian categories with noise systems. 
An additive functors $F\colon(\A, \ns) \to (\A', \ns')$ is \define{noise-bounding with constant $K \in (0, \infty)$} if there exists a non-negative real number $K\geq 0$ such that, for any $A$ in $\ob \A$, if $A\in \ns_\eps$, then $F(A)\in \ns_{K\eps}$.
The \define{category of noise systems} $\Noise$ is the category whose objects are pairs $(\A, \ns)$ consisting of an abelian category $\A$ and a noise system $\ns$, and whose morphisms are noise-bounding functors\footnote{As we are dealing with categories of large categories here, we do not expect $\Noise$ to be locally small.}.
\end{definition}

\begin{definition} \label{def_amp_cat}
Let $(\A, \alpha)$ and $(\A', \alpha')$ be abelian categories with amplitudes.
An additive functor $F\colon (\A, \alpha) \to (\A', \alpha')$ is \define{amplitude-bounding with constant $K \in (0, \infty)$} if $\alpha'(FA) \leq K \alpha(A)$ for all objects $A$ in $\ob\A$. 
The \define{category of amplitudes} $\Amp$ is the category whose objects are pairs $(\A, \alpha)$ of abelian categories and amplitudes, and whose morphisms are amplitude-bounding functors. 
\end{definition}

For our applications to stability, we will be especially interested in the case in that $\A$ is an abelian category with two amplitudes $\alpha$ and $\alpha'$, and we consider identity functor $\id_\A \colon (\A, \alpha) \to (\A, \alpha')$. 
The name is a slight abuse of notation, as this functor behaves as the identity only on part of the structure, but the symbol is more precise: the subscript contains only the category and not the amplitude.
In this context, the following is worth noting:
\begin{itemize}
    \item $\id_\A$ does not need be a morphism in $\Amp$.
    \item If $\id_\A$ is a morphism in $\Amp$, it does not need be invertible.
    \item If $\id_\A$ is invertible, it is not necessarily the case that $\alpha = \alpha'$. Rather, $\alpha'$ allows us to infer upper and lower bounds on $\alpha$, and vice versa.
\end{itemize}

We now establish the relationship between noise systems and amplitudes. 
Essentially, this relationship comes down to characterizing a norm on a vector space by the set of $\eps$-balls around the origin. 

\begin{proposition}\label{prop_noise_to_weight}
Given an abelian category $\A$ with a noise system $\ns = \{\ns_\eps\}_{\eps\geq 0}$, the function $\alpha_\ns$ on $\ob \A$ defined by
\[
\alpha_\ns(A) \coloneqq
\inf\left\{\eps \ \vert \ A\in \ns_\eps\right\},
\]
with the convention that $\inf\emptyset= \infty$, is an amplitude on $\A$.
This construction gives rise to a functor $\NA\colon\Noise \to \Amp$ defined by sending:
\begin{itemize}
    \item An object $(\A, \ns)$ of $\Noise$ to the object $(\A, \alpha_\ns)$ of $\Amp$, and
    \item A noise-bounding functor $F\colon(\A, \ns) \to (\A', \ns')$ to an amplitude-bounding functor $F\colon(\A, \alpha_\ns) \to (\A', \alpha_{\ns'})$.
\end{itemize}

\end{proposition}

\begin{proof}
All of the verifications are straightforward.
\end{proof}

Having constructed amplitudes from noise systems in \cref{prop_noise_to_weight}, we are led to consider the converse:

\begin{definition} \label{def_adjoint_noise_system}
Let $\alpha$ be an amplitude on an abelian category $\A$.
The \define{noise system associated to} $\alpha$, $\ns(\alpha)$, is defined by
\begin{equation*}
    A \in \ns(\alpha)_\eps \Longleftrightarrow \alpha(A) \leq \eps\, ,
\end{equation*}
for all $\eps \in \mathbb R_{\geq 0}$.
This construction induces a functor $\AN\colon\Amp \to \Noise$ as in \cref{prop_noise_to_weight}. 
\end{definition}

\begin{proposition}\label{prop:equ_of_cat}
The functors $\AN\colon \Amp \leftrightarrow \Noise\colon \NA$ induce an equivalence of categories. 
The natural isomorphisms of the equivalence are all given by the identity on the underlying abelian categories and are noise (amplitude) bounding for any $K >1$.
\end{proposition}

\begin{proof}
Both natural isomorphisms $\NA \circ \AN \cong \id_{\Amp}$ and $\AN \circ \NA = \id_{\Noise}$ are induced by the identity. 
The former is easily seen to fulfill $\NA(\AN(\alpha)) = \alpha$ directly. 
For the latter, note first that $\ns _{\eps} \subset \AN(\NA(\ns))_{\eps}$. 
Hence, the identity in one direction is noise bounding with constant $1$. 
Conversely, for any $A \in \AN ( \NA (\ns))_{\eps}$, we have $A \in \ns_{K \eps}$ for any $K>1$. 
Hence, the converse direction is noise bounding with constant $K >1$.
\end{proof}

In particular, we have that exact amplitude-bounding functors induce Lipschitz functions between metric spaces:

\begin{proposition}\label{prop:Weaker_Ex_Cond} 
Let $F\colon (\A,\alpha)\to (\A',\alpha')$ be an exact  additive amplitude-bounding functor with constant $K$.
Then the induced map $F\colon (\ob \A, \dist_\alpha) \to (\ob \A', \dist_\alpha')$ is $K$-Lipschitz.
\end{proposition}

\begin{proof} 
Follows directly from the definitions.
\end{proof}

As a partial converse, we can show that an additive functor that induces Lipschitz maps on metric spaces is necessarily amplitude-bounding:

\begin{proposition}\label{P:lipschitz functors}
Let $F\colon (\A,\alpha)\to (\A',\alpha')$ be an additive functor such that the induced map $F\colon (\ob \A, \dist_\alpha) \to (\ob \A', \dist_{\alpha'})$ is $K$-Lipschitz.
Then $F$ is amplitude-bounding with constant $K$.
\end{proposition}

\begin{proof}
By \cref{L:char ampl metr} we know that $\alpha'(F(A))=\dist_{\alpha'}(F(A),0)$. 
Since $F$ is additive, we have that $F0\cong0$. 
Thus, we obtain $\alpha'(F(A))=\dist_{\alpha'}(F(A),0)\leq K \dist_{\alpha}(A,0)=\alpha(A)$.
\end{proof}

\subsection{Distances from noise systems vs. amplitudes}\label{S:path vs noise}

In this section, we relate distances obtained from noise systems and amplitudes, respectively.  
Noise systems and amplitudes induce metrics on the underlying abelian categories (see \cref{def_span_metric}), and we show that passing from a noise system to its corresponding amplitude does not affect the resulting metric.
This all assembles into a commutative diagram of functors between various categories of noise systems, amplitudes, and pseudometric spaces.
\medskip

We begin by defining the relevant categories.

\begin{definition} \label{def_pseudom}
Let $\Pseudom$ be the \define{category of pseudometric spaces} and Lipschitz maps, i.e.\ $K$-Lipschitz maps for some $K \geq 0$.
\end{definition}

Here, we use the usual definition of ``pseudometric space'', except that we do not require the underlying objects of our pseudometric spaces to be sets, and we allow pseudometrics to take values in $\posrc$.
\medskip

To ensure that the functors of interest between abelian categories induce Lipschitz maps between the corresponding pseudometric spaces, we restrict our attention to exact functors, as in the following definition.

\begin{definition} \label{def_exact_cats}
We denote by $\Noise^{ex}$ be the subcategory of $\Noise$ where morphisms are exact noise-bounding functors, and by $\Amp^{ex}$ be the subcategory of $\Amp$ where morphisms are exact amplitude-bounding functors. 
\end{definition}

We are now ready to investigate the functoriality of the standard metric constructions. 
We first recall the definition of $\varepsilon$-closeness from \cite[Def. 8.4]{SC+19}: 
two objects $A$, $B$ in an abelian category $\A$ are $\eps$-close if there exist two maps $\varphi\colon A\to H$ and $\psi\colon H \to B$ such that $\ker(\varphi)\in\ns_{\eps_1}$, $\coker(\varphi)\in\ns_{\eps_2}$, $\ker(\psi)\in\ns_{\eps_3}$, and $\coker(\psi)\in\ns_{\eps_4}$, and $\eps_1+\eps_2+\eps_3+\eps_4 \leq \eps$.

\begin{proposition} \label{prop_noise_to_pseudom}
There exists a functor $\Noise^{ex} \to \Pseudom$
\begin{itemize}
    \item[(i)] Sending a pair $(\A, \ns)$ to the pseudometric space $(\ob \A, \dist_\ns)$, where $\dist_\ns$ is the metric defined by 
    \[
    \dist_\ns(A,B) = \inf \{ \varepsilon \,|\, \textrm{$A$ and $B$ are $\varepsilon$-close} \},
    \]
    where as usual $\inf \emptyset = \infty$;
    \item[(ii)] Sending an exact noise-bounding functor $F\colon(\A, \ns) \to (\A', \ns')$ to the map on objects $F\colon(\ob \A, \dist_\ns) \to (\ob \A', \dist_{\ns'})$.
\end{itemize}
Furthermore, if $F\colon(\A, \ns) \to (\A', \ns')$ is exact and noise-bounding with constant $K$, then the map on objects is $K$-Lipschitz.
\end{proposition}

\begin{proof}
The fact that $\dist_\ns$ is a pseudometric is proved, in a special case, in \cite[Prop.\ 8.7]{SC+19} and their arguments generalize immediately to the case we are considering.
Hence, the only nontrivial part of the proposition is the last claim, since this implies the functor $\Noise^{ex} \to \Pseudom$ is well-defined on morphisms.
So, suppose we are given an exact functor $F\colon(\A, \ns) \to (\A', \ns')$ which is noise-bounding with constant $K$.
Furthermore, suppose $\dist_\ns(A,B) < \eps$ for some $A, B\in \ob \A$.
Then we can choose a span $A\overset{f_0}{\longleftarrow} C \overset{f_1}{\longrightarrow} B$ with $\ker f_0 \in \ns_{\eps_1}$, $\ker f_1 \in \ns_{\eps_2}$, $\coker f_0 \in \ns_{\eps_3}$, and $\coker f_1 \in \ns_{\eps_4}$, for some $\eps_i$ satisfying $\sum_{i=1}^4 \eps_i < \eps$.
Since $F$ is noise bounding and exact, the span
$FA\xleftarrow{Ff_0} FC \xrightarrow{Ff_1} FB$ serves as a witness to the inequality $\dist_{\ns'}(FA, FB) < K\eps$.
Because this holds for any $\eps > \dist_\ns(A,B)$, we see that $\dist_{\ns'}(FA, FB) \leq K \dist_\ns(A,B)$, i.e.\ $F$ is $K$-Lipschitz.
\end{proof}

\begin{corollary}
\label{prop_weight_to_pseudom}
There exists a functor $\Amp^{ex} \to \Pseudom$
\begin{itemize}
    \item[(i)] Sending an object $(\A, \alpha)$ to the pseudometric space $(\ob \A, \dist_\alpha)$, where $\dist_\alpha$ is as defined in \cref{def_span_metric};
    \item[(ii)] Sending an exact amplitude-bounding functor $F\colon(\A,\alpha)\to(\A',\alpha')$ to the map on objects $F\colon(\ob\A,\dist_\alpha) \to (\ob\A',\dist_{\alpha'})$.
\end{itemize} 
Furthermore, if $F\colon(\A, \alpha) \to (\A', \alpha')$ is exact and amplitude bounding with constant $K$, then $F\colon(\ob \A, \dist_\alpha) \to (\ob \A', \dist_{\alpha'})$ is $K$-Lipschitz. 
\end{corollary}

\begin{proof}
By \cref{P_f_span_metric}, such a functor is uniquely given by the composition $\Amp^{ex} \xrightarrow{\AN} \Noise^{ex} \to \Pseudom$.
\end{proof}

Applying \cref{prop:equ_of_cat}, we obtain:

\begin{corollary} \label{cor_commutative diagram}
The functors of \cref{prop_noise_to_pseudom}, \cref{prop_weight_to_pseudom}, \cref{prop_noise_to_weight}, and \cref{def_adjoint_noise_system} assemble into on-the-nose commutative triangles
\[
\begin{tikzcd}
\Amp^{ex} \arrow[rd] \arrow[rr, "\AN" ] & & \Noise^{ex} \arrow[ld] \\
 & \Pseudom &
\end{tikzcd}
,
\begin{tikzcd}
\Noise^{ex} \ar[dr] \ar[rr, "\NA"] & & \Amp^{ex} \ar[dl] \\
& \Pseudom &
\end{tikzcd}
.
\]
\end{corollary}

We end this section with a remark on another class of metrics for homological algebra setting which was defined in \cite{biran2021triangulation}.

\begin{remark}
    In \cite{biran2021triangulation} several notions of metrics for triangulated (persistence) categories were constructed.
    These rely on the additional data of specifying a class of objects $\mathcal{F}$. 
    Instead of assigning weights to morphisms, these distances assign weights to triangles. 
    In particular, the metrics $\underline{d}_{\mathcal F}$ of \cite[Rem. 2.0.5]{biran2021triangulation} may seem conceptually very close to amplitude distances.
    Roughly, $\underline{d}_{\mathcal F}$ measures how many exact triangles with lower tip in $\mathcal F$ it takes to transform one object into another, and weights this value by the weight of the triangles. At first glance, this makes amplitude distances appear to be a special case of the distances in \cite{biran2021triangulation}. 
    Given an abelian category $\mathcal{A}$ with amplitude $\alpha$, the obvious route of comparison would be to consider the (bounded) derived category of $\mathcal{A}$. 
    One may then take $\mathcal{F} = \{ C \mid \sum_{i \in \mathbb Z} \alpha(H_i(C)) < \infty \}$
    and assign to each exact triangle $T$
    \begin{center}
        \begin{tikzcd}
            A \arrow[rr] & & B \arrow[ld] \\
            & C \arrow[lu] &
        \end{tikzcd}
    \end{center}
    the weight $\omega(T) = \sum_{i \in \mathbb Z} \alpha(H_i(C)).$ \\
    We note, however, that these types of weights do not fulfill the axioms of \cite[Def. 2.0.1]{biran2021triangulation}. In particular, they fail normalization.  
    Furthermore, the definition of $\underline{d}_F$, in \cite{biran2021triangulation} requires all triangles in a sequence to be oriented in the same way, and then symmetrize the metric by taking a maximum. 
    This roughly corresponds to setting \begin{align*}
        \dist(A,B) = \max &
        \Bigl\{
        \inf \{ \alpha(\ker f) + \alpha( \coker f) \mid f \in \Hom_{\mathcal A}(A,B)\},
        \\
        & 
        \inf \{ \alpha(\ker f) + \alpha( \coker f )\mid f \in \Hom_{\mathcal A}(B,A)\}  
        \Bigl\},
    \end{align*}
    which differs from what we defined. For instance, in the case of $1$-parameter persistence modules and the shift amplitude, this value will be infinite for any two interval modules that are not isomorphic. 
    
    Additionally, in \cite[Sec. 4]{biran2021triangulation} the authors define metrics that take persistence into account by adding an additional persistence structure to the triangulated category and generalizing from exact triangles to triangles which are exact up to an $\varepsilon$-isomorphism. However, in these scenarios, they only weigh such almost exact triangles by the necessary amount of shifting. 
    In particular, there seems to be no way to represent arbitrary amplitude distances in this framework, and therefore, although their approach seems close to ours, it is different in practice.
\end{remark}

\section{Persistence modules and finiteness conditions}\label{S_permod_finiteness}

In this section, we first cover some finiteness notions for multiparameter persistence modules and then we discuss some properties of multiparameter persistence modules on intervals.

\begin{notation}
We denote by $\Vect$ the abelian category of finite-dimensional vector spaces over some fixed field $\mathbb{F}$. 
We use $\Pos$ to denote an arbitrary poset, and by abuse of notation, we denote by $\Pos$ also the canonical posetal category associated with $\Pos$.
\end{notation}

\begin{definition}\label{D:per mod}
A \define{persistence module} is a functor $\Pos\to \Vect$. 
When $\Pos=P^n$ for $P$ a totally ordered set and $n\geq 1$  a natural number, we also call such a functor
an \define{$n$-parameter persistence module}. 
Furthermore, when $n\geq 2$, we alternatively call such a module a \define{multiparameter persistence module}.
\end{definition}

Persistence modules are one of the main objects of study in TDA: by computing simplicial homology with coefficients in a field of an $n$-parameter family of simplicial complexes, one obtains for each homology degree an $n$-parameter persistence module.

\begin{remark}\label{rem:module_notation}
We will constantly make use of the fact that persistence modules over $\Pos =\mathbb Z^n$, $\mathbb R^n$, 
or more generally any monoidal poset, admit an equivalent formulation as algebraic objects. 
In Definition \ref{D:per mod} we define persistence modules  as functors $\Pos\to \Vect$. 
The category with objects such as functors, and morphisms given by natural transformations between them, can be easily seen to be isomorphic to the category with objects $\Pos$-graded modules over the monoidal ring $\Field[\Pos_{\geq 0}]$, and morphisms given by with homogeneous module homomorphisms. 
Using the latter perspective, we will frequently denote the application of $M(u \leq v)$ to an element $m \in M_u$ in the form 
\[ 
x^{u-v}m \coloneqq M(u \leq v)m \, ,
\]
for $M \in \Vect^{\Pos}$, $u\leq v \in Q$.
\end{remark}

To balance computability and generality, we focus on persistence modules that admit a \textit{finite encoding} \cite{Miller-hom}. 
Roughly, a finitely encoded module is a module that can be described using a finite poset (\cref{D:finite_enc}). 
However, this level of generality is not entirely suitable to our setting since the category of finitely encoded persistence modules is not abelian \cite[Ex. 4.25]{Miller-hom}. 
Instead, as a consequence of \cite[Thm 1.1]{Lukas}, we restore abelianity by only allowing for encodings of certain types. 
Particular attention will be paid to the class of staircase modules (\cref{ex:staircase}).

\begin{figure}[h!]
\centering
\subfloat[A staircasemodule that is not finitely presented.]{%
\begin{tikzpicture}[scale=1,y=3cm, x=3cm,>=latex',font=\sffamily]
\begin{scope}
\draw[->, thick] (0,0) -- (0,1);
\draw[->, thick] (0,0) -- (1,0);
\coordinate (fe1) at (0.3,0.2);
\coordinate (fe2) at (0.8,0.2);
\coordinate (fe3) at (0.8,0.45);
\coordinate (fe4) at (0.6,0.45);
\coordinate (fe5) at (0.6,0.65);
\coordinate (fe6) at (0.3,0.65);
\draw[-, gray] (0.3,-0.03) -- (0.3,0.03);
\draw[-, gray] (0.8,-0.03) -- (0.8,0.03);
\draw[-, gray] (0.6,-0.03) -- (0.6,0.03);
\draw[-, gray] (-0.03,0.2) -- (0.03,0.2);
\draw[-, gray] (-0.03,0.45) -- (0.03,0.45);
\draw[-, gray] (-0.03,0.65) -- (0.03,0.65);
\fill[Bsugarpaper] (fe1) rectangle (fe3);
\fill[Bsugarpaper] (0.3,0.4) rectangle (fe5);
\draw[-, Bdust, line width=0.5mm] (fe2) -- (0.8,0.458);
\draw[-, Bdust, line width=0.5mm] (fe3) -- (fe4);
\draw[-, Bdust, line width=0.5mm] (0.6,0.442) -- (0.6,0.658);
\draw[-, Bdust, line width=0.5mm] (fe5) -- (fe6);
\node at (0.3,-0.06) {{\tiny 1.5}};
\node at (-0.06,0.2) {{\tiny 1}};
\node at (0.6,-0.06) {{\tiny 3}};
\node at (-0.07,0.45) {{\tiny 2.5}};
\node at (0.8,-0.06) {{\tiny 4}};
\node at (-0.06,0.65) {{\tiny 3}};
\end{scope}
\end{tikzpicture}
\label{fig:cub_enc}}
$\qquad $
\subfloat[A finitely encoded module that is not a staircase module.]{%
\begin{tikzpicture}[scale=1,y=3cm, x=3cm,>=latex',font=\sffamily]
%finitely encoded
\pgfdeclarelayer{background}
\pgfsetlayers{background,main}
%axis
\draw[->, thick] (0,0) -- (0,1);
\draw[->, thick] (0,0) -- (1,0);
%x
\draw[-, gray] (0.2,-0.03) -- (0.2,0.03);
\draw[-, gray] (0.83,-0.03) -- (0.83,0.03);
%y
\draw[-, gray] (-0.03,0.2) -- (0.03,0.2);
\draw[-, gray] (-0.03,0.65) -- (0.03,0.65);
%labels
\node at (0.2,-0.06) {{\tiny 1}};
\node at (-0.06,0.2) {{\tiny 1}};
\node at (0.83,-0.06) {{\tiny 4}};
\node at (-0.06,0.65) {{\tiny 3}};

\begin{scope}[blend mode=multiply]
\node (fe1) at (0.2,0.7) {};
\node (fe2) at (0.8,0.2) {};
\node (fe6) at (0.55,0.55) {};
\begin{pgfonlayer}{background}
    \fill[bend left, fill=Bsugarpaper]
    (fe2.east)
      to [out=360,in=130] (fe1.south)
      to (fe6.south)
      to (fe2.east);
\end{pgfonlayer}
\end{scope}
\end{tikzpicture}
\label{fig:fin_enc}}
\caption{Examples illustrating the difference between staircase module and finitely encoded module. 
To every point is associated a copy of the field $\mathbb{F}$ or $0$ if it is in or outside a blue region, respectively. 
The thick line denotes a closed boundary, and the lack of it an open boundary. 
The structure morphisms are given by the identity on $\mathbb{F}$ whenever possible.}
\label{F:finitess}
\end{figure}

We begin by introducing some general notation and language.

\begin{definition}\label{D:upset_downset}
An \define{upset} of a poset $\Pos$ is a subset $U\subseteq \Pos$ such that for any $u\in U$ and $q\in \Pos$ whenever $u\leq q$ then $q\in U$. 
Similarly, a \define{downset} of a poset $\Pos$ is a subset $D\subseteq \Pos$ such that for any $d\in D$ and $q\in \Pos$ whenever $d\geq q$ then $q\in D$.
An \define{interval} in $\Pos$ is the intersection of an upset and a downset of $\Pos$.
\end{definition}

Note that, according to this definition, an interval can be disconnected.

\begin{notation} For a poset $\Pos$, and $p,q \in \Pos$, we denote
\begin{enumerate}
    \item The upset given by all elements greater or equal to $q$, by $[q, \infty)$.
    \item The downset given by all elements lesser or equal to $q$, by $(-\infty, p]$.
    \item The interval given by the intersection of $[p, \infty)$ and $(- \infty, q]$ by $[p,q]$.
\end{enumerate}
For $A\subset B$ subsets of a poset $\Pos$, we write $ A \subsetl B$ to indicate that $A$ is a downset in $B$ and $A \subsetr$ to indicate that $A$ is an upset in $B$.
\end{notation}

\begin{definition}
    Recall that an \define{algebra} on a set $S$, is a subset of the power set of $S$, which contains $S$ and is closed under unions and complements.
\end{definition}

As already mentioned above, it is useful to have a general framework that specifies what class of persistence modules one allows for. 
This is taken care of by the following definition:

\begin{definition}
Let $\Pos$ be a poset. An \define{encoding structure} $\mathfrak X$ on $\Pos$ is a subset of the powerset of $\Pos$, such that
\begin{enumerate}
    \item $\mathfrak X$ is an algebra (i.e. closed under finite unions and complements, and contains $\Pos$);
    \item Every element of $\mathfrak X$ is a finite union of intervals, which are themselves contained in $\mathfrak X$;
    \item If $I \in \mathfrak X$ is an interval of $\Pos$, then there exist upsets $U,V \in \mathfrak X$, such that $I = U \cap V^c$;
\end{enumerate}
\end{definition}

\begin{example} \label{ex:staircase}Let $\mathbb T$ be a totally ordered set.
A \define{cube} in $\mathbb T^n$ is a subset of the form $I_1 \times \dots \times I_n$, where the $I_j$ are intervals in $\mathbb T$.
A finite union of cubes that are unbounded above is called a \define{staircase}. 
The set 
\[ 
\mathfrak{C} \coloneqq \{ S \subset \mathbb T^n \mid S \textnormal{ is a finite union of cubes in $\mathbb T^n$ }\} \ 
\] 
is an encoding structure, which we call the \define{staircase structure $\mathbb T^n$}, as the upsets of $\mathfrak C$ are precisely the staircases. We denote by $\cCubes$ the subalgebra of $\mathfrak{C}$ generated by the closed staircases. 
Again, $\cCubes$ is an encoding structure on $\mathbb{T}^n$. 

Furthermore, observe that the intervals of $\mathfrak C$ are all given by finite unions of cubes in $\mathbb T^n$ and that the intervals of $\cCubes$ are given by finite unions of cubes $I_1 \times \dots \times I_n$ where the $I_j$ are closed at the left and open at the right.
\end{example}

\begin{example}\label{ex:more_general_classes_of_enc_struct}
    The algebra generated by piecewise linear upsets, by semi-algebraic upsets, or alternatively by any o-minimal structure in the sense of \cite{van1998tame}, defines an encoding structure on $\mathbb R^n$. 
    Similarly, one may consider the subalgebra of the latter examples generated by only the topologically closed upsets, called the closed below encoding structures.
\end{example}

\begin{definition} 
An \define{encoding} of an object $M\in\ob\Vect^{\Pos}$ by a poset $\Pos'$ is a morphism $e\colon \Pos\to \Pos'$ together with an object $M'\in\ob\Vect^{\Pos'}$ and an isomorphism $\eta\colon M \xrightarrow{\sim} e^{\ast}M'$ between the object and the pullback of $M'$  along $e$.
Moreover, the encoding is \define{finite} if $\Pos'$ is finite and $M'$ is pointwise finite-dimensional.
Let $\mathfrak X$ be an encoding structure on $\Pos$ and $e\colon \Pos \to \Pos'$, $\eta\colon M \xrightarrow{\sim} e^{\ast}M'$ a finite encoding of $M$. 
If $e$ is such that for every upset $U \subset \Pos'$, $e^{-1}(U)$ is an element of $\mathfrak X$, or equivalently that $e^{-1}(p) \in \AlgInt$, for every $p \in \Pos'$, then we say that the encoding is \define{of class $\mathfrak{X}$} or an \define{$\mathfrak{X}$-encoding}.
\end{definition}

\begin{definition}\label{D:finite_enc}
An object $M$ in $\ob\Vect^{\Pos}$ is \define{finitely encoded} if it admits a finite encoding.
Moreover, $M$ is \define{$\mathfrak{X}$-encoded}, or \define{encoded by $\mathfrak X$}, if it admits a $\mathfrak X$-encoding. 
If $\Pos = \mathbb T^n$ for a totally ordered poset $\mathbb{T}$, and $\mathfrak X = \mathfrak C$ is the staircase structure, we say that $M$ is a \define{staircase module}.
\end{definition}

\begin{notation}
For a poset $\Pos$ and an encoding structure $\mathfrak X$ on $\Pos$ we denote by $\permodq{\Pos}$ the full subcategory of $\Vect^\Pos$ given by finitely encoded persistence modules,
and by $\permodx$ the full subcategory given by the objects that are finitely $\mathfrak X$-encoded.
Note that, when $n=1$, the categories $\permodq{\mathbb R^n}$ and $\permod$ coincide.
\end{notation}

\begin{example}\label{ex:encoding_intervals}
Let $I \subset \mathcal Q$ be an interval of an encoding structure $\mathfrak X$ on $\Pos$. 
We denote $\Field[I]$, the \define{interval module on }$I$, the module obtained as follows. 
Consider the finite poset with two elements $0$ and $1$ and the unique non-trivial relation given by $0\leq 1$. 
We denote this by $\{0\leq 1\}$, and similarly we denote by $\{0 \leq 1\}^2$ the poset given by the cartesian product $\{0,1\}^2$ with partial order $(a,b)\leq (a',b')$ if and only if $a\leq a'$ and $b\leq b'$. 

Consider the map of partially ordered sets $e\colon\Pos \to \{0 \leq 1\}^2$ induced by the two maps $e_D\colon \Pos \to \{ 0 \leq 1\}$ and $e_U\colon \Pos \to \{ 0 \leq 1\}$ defined as follows:
\begin{equation*}
    e_D(q) \coloneqq 
    \begin{cases}
      0 \ \textrm{ if } q \in D \\
      1 \ \textrm{ otherwise }
    \end{cases}\qquad
    e_U(q) \coloneqq 
    \begin{cases}
      1 \ \textrm{ if } q \in U \\
      0 \ \textrm{ otherwise }\, .
    \end{cases}
\end{equation*}
Let $M' \in \Vect^{\{0 \leq 1\}^2}$ be the unique persistence module with $M'_{(0,1)}=\Field$ and $0$ otherwise. 
Finally, define $\Field[I] = e^*M'$.
By construction, $\Field[I]$ is $\mathfrak X$-encoded. Explicitly, $\Field[I]$ is given by $\Field$ on $I$, $0$ everywhere else, and with identity as structure morphism for all relations in $I$. 
\end{example}

\begin{figure}[ht!]
\centering
\begin{tikzpicture}[scale=1,y=3cm, x=3cm,>=latex',font=\sffamily]
%axis
\draw[->, thick] (0,0) -- (0,0.75);
\draw[->, thick] (0,0) -- (0.75,0);
%thicks
\draw[-, gray] (0.2,-0.03) -- (0.2,0.03);
\draw[-, gray] (0.6,-0.03) -- (0.6,0.03);
\draw[-, gray] (-0.03,0.1) -- (0.03,0.1);
\draw[-, gray] (-0.03,0.5) -- (0.03,0.5);
% %labels
\node at (0.2,-0.06) {{\tiny 2}};
\node at (-0.06,0.1) {{\tiny 1}};
\node at (0.6,-0.06) {{\tiny 6}};
\node at (-0.06,0.5) {{\tiny 5}};
\begin{scope}[blend mode=multiply]
%points
\coordinate (x1) at (0.2,0.1);
\coordinate (x2) at (0.6,0.1);
\coordinate (x3) at (0.2,0.5);
\coordinate (x4) at (0.6,0.5);
\coordinate (u3) at (0.9,0.85);
\coordinate (d3) at (-0.37,-0.22);
%rectangle M
\fill[gray!20] (x1) rectangle (u3);
\fill[gray!20] (d3) rectangle (x4);
\fill[Bsugarpaper] (x1) rectangle (x4);
\node (b1) at (0.65,0.65) {$\bullet$};
\node (b2) at (-0.02,-0.13) {$\bullet$};
\node (b3) at (0.1,0.55) {$\bullet$};
\node (b4) at (0.45,0.2) {$\bullet$};
\node at (0.75,0.75) {$(1,1)$};
\node at (-0.2,-0.13) {$(0,0)$};
\node at (0.15,0.65) {$(1,0)$};
\node at (0.62,0.2) {$(0,1)$};
\draw[->] (b4) -- (b1);
\draw[->] (b2) -- (b4);
\draw[->] (b2) -- (b3);
\draw[->] (b3) -- (b1);
\end{scope}
\end{tikzpicture}
\caption{An example of an interval module, defined over $I=D\cap U$, where $D=\{x\in\mathbb{R}^2 \, \vert \, x\leq (6,5)\}$ and $U=\{x\in\mathbb{R}^2\, \vert \, x\geq (2,1)\}$, with the bullets representing the encoding poset $\{0\leq 1\}^2$. The sets $U$ and $D$ are depicted in gray.}
\label{fig_ex_encoding}
\end{figure}

A useful property of finitely encoded modules is that there always exists a common finite encoding.

\begin{lemma}\label{L:common_encoding}
Let $M_1,\dots, M_n$ be finitely encodable persistence modules over $\Pos$. 
Then there exist an encoding map $e\colon \Pos \to \Pos'$ which is part of an encoding for $M_1, \dots, M_n$ simultaneously. 
Furthermore, if $M_1, \dots, M_n$ are $\mathfrak X$ encodable, for some encoding structure $\mathfrak X$ on $\Pos$, then $e$ can also be taken of class $\mathfrak X$.
\end{lemma}

\begin{proof}
We only show the case $n=2$. 
The general case is entirely analogous.
Let $e_M\colon \Pos\to \mathcal{P}_M$ and $e_N\colon \Pos\to \mathcal{P}_N$ be finite encodings of $M$ and $N$, respectively, with modules $M'$ and $N'$. 
Consider the induced map
\begin{alignat}{2}
e\colon \Pos&\notag\longrightarrow \mathcal{P}_M\times \mathcal{P}_N \\
q&\notag\mapsto (e_M(q),e_N(q))\, .
\end{alignat}
Define two modules $M''$ and $N''$ as the pullback along the canonical projections $\pi_M\colon \mathcal{P}_M\times \mathcal{P}_N\to \mathcal{P}_M$, $\pi_N\colon \mathcal{P}_M\times \mathcal{P}_N\to \mathcal{P}_N$ of $M'$ and $N'$, respectively.
It follows that there exist two isomorphisms $M\cong e^\ast M''$ and $N\cong e^\ast N''$, proving the claim. 
Furthermore, the fibers of $e$ are, by construction, given by the pairwise intersections of the fibers of $e_M$ and $e_N$. 
Hence, if the latter two are of class $\mathfrak X$, so is $e$.
\end{proof}

\begin{remark}\label{rmk_encoding_disjoint_intervals}
Let $I \subset \Pos$ be an interval. Then, for any encoding map $e\colon \Pos \to \Pos'$ of $\Field[I]$, $I$ is saturated under $e$, i.e. we have $e^{-1}(e(I)) = I$. To see this, fix some module $M$ over $\Pos'$ such that $e^{*}(M) \cong \Field[I]$. 
Then, by construction, $I = e^{-1}(S)$ where $S$ is the support of $M$, and thus $I$ is saturated.
A convenient consequence of this property is that intersections of intervals are compatible under common encodings, in the sense that
\[
e( \bigcap I_j) = \bigcap e(I_j),
\]
whenever $e \colon \Pos \to \Pos'$ is a common encoding map for a family of interval modules $\Field[I_j]$.
\end{remark}

As an immediate consequence of Lemma \ref{L:common_encoding} and the additivity of pullback functors, we obtain that the sum of two finitely encoded persistence modules (of a certain class $\mathfrak{X}$) is again finitely encoded (and of class $\mathfrak{X}$). 
In other words, we have:

\begin{corollary}
For any encoding structure $\mathfrak{X}$ on $\Pos$, the category $\permodx$ is an additive subcategory of $\Vect^{\Pos}$.
\end{corollary}
For our purposes, we are particularly interested in the cases where $\permodx$ is a full {\em abelian} subcategory of $\Vect^\Pos$. 
Even if, as we have already recalled, it is in general not true, under certain connectivity assumptions, which are fulfilled for the algebra of cubes and certain subalgebras of the PL and semialgebraic sets, this turns out to be the case. 
We provide now a brief introduction to the topic and refer to \cite{Lukas} for more details and the full proofs.

\begin{definition}
We say that a poset $\Pos$ is \define{$\leq$-connected} if it is connected as a category. 
In other words, if, for every $q,q' \in \Pos$, there exists a zigzag
\[ 
q \lessgtr q_1 \lessgtr  \dots \lessgtr  q_k \lessgtr  q'
\] 
of inequalities between the two elements, where $\lessgtr$ denotes either $\leq$ or $\geq$. 
The maximal $\leq$-connected subsets of $\Pos$ are called the \define{$\leq$-components} of $\Pos$. 
We denote the set of these by $\pi_0(\Pos)$.
\end{definition}

\begin{remark}\label{ex:cube-comp}
Every cube in $C \subset \mathbb R^n$ is $\leq$-connected. 
To see this, note that a zigzag between any two elements $x=(x_1,\dots,x_n)$ and $y=(y_1,\dots,y_n)$ of the cube is given by the elements 
\[
x\lessgtr (y_1,x_2\dots,x_n) \lessgtr (y_1, y_2,x_3 \dots, x_n) \lessgtr \dots\lessgtr  (y_1,\dots,y_{n-1}, x_n) \lessgtr y \, .
\]
As a cube is a product of subsets of $\mathbb R$, all of these elements are again in $C$.
\end{remark}

More generally, every element of the closed below encoding structures described in \cref{ex:more_general_classes_of_enc_struct} has finitely many $\leq$-connected components.

We now describe how the connected components of intervals of an encoding structure $\mathfrak X$ influence the categorical properties of $\permodx$.

\begin{definition}
An encoding structure $\mathfrak X$ on a poset $\Pos$ is \define{connective} if every interval $I$ in $\mathfrak X$ has only finitely many $\leq$-components and these components are again elements of $\mathfrak X$. 
\end{definition}

\begin{remark}\label{rmk:C_conn}
The encoding structure of cubes $\mathfrak C$ is connective. 
Indeed, let $I \subset \mathbb R^n$ be an interval given by a finite union of cubes. 
First, note that we can write $I$ as a finite disjoint union of cubes $C_i$. 
By Example \ref{ex:cube-comp}, each of these cubes is $\leq$-connected. 
Hence, all of the $\leq$-components of $I$ can be written as a finite disjoint union of $C_i$. 
As the components are pairwise disjoint and there are only finitely many $C_i$, there can only be finitely many $\leq$-components.
\end{remark}

For connective encoding structures $\mathfrak X$, the category $\permodx$ is a full abelian subcategory of $\Vect^{\Pos}$ (Theorem \ref{thm:X_conn_abelian}).
However, the analogous argument as in Remark \ref{rmk:C_conn} does not work in the general piecewise linear or semialgebraic setting. 
In fact, the antidiagonal $\{(x,y) \mid x= -y\}$ already is a PL interval with infinitely many $\leq$-connected components.
For many applications, it can be of interest to allow for PL or semialgebraic encodings (compare \cite{Miller-hom}) while still retaining the abelianity of the underlying category. 

\begin{example}\cite[Rem. 3.18]{Lukas}
    For any of the encoding structures described in \cref{ex:more_general_classes_of_enc_struct}, the closed below encoding structure is connective.
\end{example}

For connective encoding structures, we have the following theorem:

\begin{theorem}(\cite[Thm. 3.4]{Lukas})
\label{thm:X_conn_abelian}
Let $\mathfrak X$ be a connective encoding structure on the poset $\Pos$. 
Then the category $\permodx$ is a full abelian subcategory of $\Vect^{\Pos}$.
\end{theorem}

\begin{corollary}
The category $\permod$ is a full abelian subcategory of $\Vect^{\mathbb{R}^{n}}$.
\end{corollary}

The connective components of an interval are the key part to understanding morphisms between interval modules (see \cite{BauerLesnick14} for many examples). 
We will make extensive use of this understanding in \cref{S_classification_amplitudes}.
As a direct consequence of \cite[Prop. 3.10]{Miller-hom}, morphisms between interval modules can be explicitly classified as follows.

\begin{corollary}\label{prop:mor_between_int_mod} 
Let $I, I' \subset \Pos$ be intervals, $C \in \pi_0(I \cap I')$ and $p,q \in C$. 
Set $S \coloneqq \{ J \in \pi_0(I \cap I') \mid J \subsetl I, J \subsetr I' \} $.
Then 

\begin{enumerate}
    \item for every morphism $f \colon \Field[I] \to \Field[I']$, we have $f_p = f_q$ for every $p,q\in\Pos$. Hence, we can denote by $f_C$ the canonical element of $\Field$ specifying $f_p\colon \Field \to \Field$;
    \item The linear map $\Hom_{\permodq{\Pos}}(\Field[I],\Field[I']) \to \Field^S$ that associates $f \mapsto (f_C)_{C \in S}$ is an isomorphism. 
\end{enumerate} 

\end{corollary}
\begin{definition}\label{def_cat_intervals_encoding}
Let $\mathfrak{X}$ be an encoding structure on a poset $\Pos$. 
The \define{category of encoding intervals} $\AlgInt$ has as objects intervals $I$ in $\mathfrak X$ and morphisms given by  
\begin{equation*}
\Hom_{\AlgInt}(I,I') \coloneqq 
\{J \subset I \cap I' \mid J \in \AlgInt, J \subsetl I , J \subsetr I'  \}
\end{equation*}
with $\mathrm{id}_{I} = I$ and composition given by $J' \circ J \coloneqq J \cap J'$.
\end{definition}

\begin{remark}\label{rem:int_mod_func}
One can use \cref{prop:mor_between_int_mod} to make the construction in \cref{ex:encoding_intervals} functorial:
\begin{align*}
    \Field[-] \colon \AlgInt & \to \permodx \\
    J \subset I \cap I' & \mapsto g
\end{align*}
where $g$ is the unique morphism in $\Hom_{\permodq{\Pos}}(\Field[I],\Field[I'])$ specified by $(f_C)_{C \in S} \in \Field^{S}$, with $f_C = 1$ if $C\subset J$ and zero otherwise. 
\end{remark}

\begin{remark}
Recall that a subquotient of $A \in \mathcal{A}$, in an abelian category $\mathcal A$, is an object $B$ together with an epimorphism $f \colon A \twoheadrightarrow C$ and a monomorphism $ B \hookrightarrow  C$. 
We say that $B$ admits the structure of a subquotient of $A$ if such a pair of arrows exists. 
By considering the pullback square 
    \[
    \begin{tikzcd}
        C' \arrow[d, two heads] \arrow[r, hook] & A \arrow[d, two heads] \\
        B  \arrow[r, hook]& C 
    \end{tikzcd}
    \]
the cospan can be equivalently replaced by a span given by a monomorphism $C'\hookrightarrow A$ and an epimorphism $C' \twoheadrightarrow B$.
\end{remark}
\cref{rem:int_mod_func} induces the following dictionary between algebraic and poset language. 

\begin{corollary}\label{cor:intervals_vs_modules}
Let $\mathfrak X$ be a connective encoding structure, $I \in \AlgInt$, and $M \in \permodx$. 
We denote by $\mathrm{sub}(M),\mathrm{quot}(M)$ and $\mathrm{subquot}(M)$ the isomorphism classes of subobjects, quotient objects, and subquotient objects of $M$, respectively. 
The functor $I \mapsto \Field[I]$ described in \cref{rem:int_mod_func} together with \cref{prop:mor_between_int_mod} induces natural bijections 
        \begin{align*}
            \{ J \in \AlgInt \mid J \subsetr I \} &\cong \mathrm{sub}(\Field[I]) \\
            \{ J \in \AlgInt \mid  J \subsetl I \} &\cong \mathrm{quot}(\Field[I]) \\
              \{ J \in \AlgInt \mid J \subset I \}&\cong \mathrm{subquot}(\Field[I]).
        \end{align*}
Furthermore, unions can be translated into sums and extensions as follows. 
Let $I_0, I_1 \in \AlgInt$. 
Then
    \begin{enumerate}
       \item $I = I_0 \cup I_1$, for $I_0, I_1 \subsetr I$, if and only if there exists an epimorphism \[ \Field[I_0] \oplus \Field[I_1] \twoheadrightarrow \Field[I] \] which is componentwise a monomorphism.
        \item $I = I_0 \cup I_1$, for $I_0, I_1 \subsetl I$, if and only if there exists a monormophism \[ \Field[I] \hookrightarrow \Field[I_0] \oplus \Field[I_1]  \] which is componentwise an epimorphism.
        \item $I = I_0 \cup I_1$, for $I_0 \subsetr I$, $I_1 \subsetl I$ with $I_0 \cap I_1 = \emptyset$ if and only if there is a short exact sequence \[ 0 \to \Field[I_0] \to \Field[I] \to \Field[I_1] \to 0 . \]
    \end{enumerate}
\end{corollary}

Finally, we describe how one may restrict and extend finitely encoded persistence modules from and to intervals. 
This result will be an important ingredient for our investigation of when two persistence modules of distance $0$ are isomorphic in \cref{S:discriminativity} (see \cref{lem:poinwise_iso_lemma}).

\begin{proposition}\label{prop:extension_by_zero}
Let $I \in \mathfrak X$. 
Denote by $\mathfrak X|_{I}$ the algebra on $I$ given by elements of the form $S \cap I$, for $S \in \mathfrak X$. 
Then the following holds:
    \begin{enumerate}
    \item $\mathfrak X|_{I}$ is an encoding structure on $I$.
    \item If $\mathfrak X$ is connective, then so is $\mathfrak X|_{I}$.
    \item The extension by $0$ functor 
        \begin{align*}
        -|^{\Pos} \colon \permodx[ I][\mathfrak X|_{I}] &\to \permodx \\
        M &\mapsto \{ p \mapsto \begin{cases}
            M_p & p \in I \\
            0 & p \notin I
            \end{cases}\, ,
        \end{align*}
    with the obvious structure morphisms and definition of morphism, is well-defined and fully faithful. If $\mathfrak X$ is connective, then $-|^{\Pos}$ defines an exact fully faithful embedding of abelian categories.
    \item For $M \in \permodx$, the module $(M|_{I})|^{\Pos}$ is naturally a subquotient of $M$.
    \end{enumerate}
\end{proposition}

\begin{proof}
1. and 2. are immediate from the fact that every interval of $I$ is also an interval of $\Pos$.

To show the remaining two points, we restrict without loss of generality to the case where $I$ is either an upset or a downset in $\mathfrak X$. 
Indeed, the general case is a composition of the latter two. 
By dualizing, it suffices to show the case where $I$ is an upset. Denote by $i$ the inclusion $I \hookrightarrow \Pos$.
Consider the left adjoint to the restriction functor $i^* \colon \Vect^{\Pos}\to \Vect^{I}$, namely $i_! \colon \Vect^{I} \to \Vect^{\Pos}$.
$i_!$ is given by left Kan-extension, i.e. we have $(i_!M)_p = \varinjlim_{q \leq p, q\in I} M_q$, with the induced structure morphisms. 
Since $I$ is an upset, it follows that $(i_!M) = M^{\Pos}$, as defined above. 
The counit of adjunction $i_! \colon i^*M \to M$ is given by the identity at $p \in I$ and by $0 \hookrightarrow M_p$ otherwise. 
In particular, it is given by a monomorphism and $(M|_{I})|^{\Pos}$ is naturally a submodule of $M$, which shows point 4. 
The unit map $N \to i^* i_!N$ is an isomorphism, indeed for every $p \in I$ the indexing category $\{q \leq p, q\in I\}$ has terminal object $p$. 
It follows that $i_!$ is fully faithful. 
By construction, it is exact.
It remains to see that $i_!$ restricts to a functor of the subcategories of $\mathfrak X|_{I}$ and $\mathfrak X$ encodable submodules. 
Let $M \in \permodx[ I][\mathfrak X|_{I}]$ with an encoding map $e \colon I \to \Pos'$, and encoding module $N \in \permodq{\Pos'}$. 
We extend $e$ to a map of posets $e_{-\infty}\colon\Pos \to \Pos'_{- \infty}$, by sending every point not in $I$ to $- \infty$. 
Since $I$ is an upset, this construction defines morphism of posets, which has fibers in $\mathfrak X$.
Denote by $N_{-\infty}$ the extension of $N$ by $0$ at $- \infty$.
Then 
\[
M|^{\Pos} \cong (e^* N)|^{\Pos} \cong e_{-\infty}^*(N_{-\infty}) \, . 
\]
It follows that $M|^{\Pos}$ is $\mathfrak X$ encoded, as was to be shown.
\end{proof}

\section{Amplitudes and amplitude distances in TDA}\label{S:stability}
It turns out that most of the distances classically relevant in topological data analysis fit into the framework of amplitudes, at least up to equivalence of metrics. 
Most prominently, this includes the Bottleneck distance and the Wasserstein distance, and their multiparameter variations, the generalized interleaving distances and measure distances of \cite{Berkouk_2021,Bubenik2014,BSS18,Lesnick2015,skraba_turner}. 
Working with distances in the framework of amplitudes has two important advantages. 
Firstly, amplitudes give a ready-to-use way to construct new distances, which capture specific features of a persistence module, depending on the choice of amplitude. 
For example, this allows one to define a family of distances based on the $k$-longest bars amplitudes of \cite{K19_Tropical}, which interpolates between the Bottleneck and the $1$-Wasserstein distances. 
Secondly, after one has interpreted two distances in the language of amplitudes, the results of \cite{skraba_turner} give an easy way of comparing these distances by investigating the relations between their amplitudes and cost functions. 

The goal of this section is to point out the most prominent examples of amplitude distances that occur in TDA and some of the relations between them that arise directly from the relations of their amplitudes. 

\subsection{Interleaving distance and shift amplitudes}
Arguably, the most used on persistence theory is the interleaving distance \cite{Bubenik2014,isometry1,Lesnick2015}: Given a persistence module $M \in \Vect^{\mathbb R^n}$ and $\varepsilon \in [0,\infty)$ we denote by $M_{+\eps}$ the persistence module obtained by shifting $M$ by the vector $\vv{\varepsilon} = \varepsilon (1, \dots, 1)$, i.e. we set $(M_{+\varepsilon})_p \coloneqq M_{p + \vv{\varepsilon}}$ with the obvious structure morphisms inherited from $M$. 
This construction induces an action of $[0, \varepsilon)$ on the category $\Vect^{\mathbb R^n}$. 
We denote by $\eta_{\varepsilon}\colon M \to M_{+\varepsilon}$ the natural transformation given by $m \mapsto x^{\vv{\varepsilon}}m$.
\begin{definition}
    An \define{$\eps$-interleaving} between $M$ and $N$ in $\Vect^{\mathbb R^n}$ is given by morphisms $f\colon M\to N_{+\eps}$ and $g\colon M\to N_{+\eps}$ in $\permod$, such that $g_{+ \varepsilon}\circ f=\eta_{2\varepsilon,M}$ and $f_{+ \varepsilon}\circ g= \eta_{2\varepsilon,N}$.
	The \define{interleaving distance} between two objects $M$ and $N$ in $\Vect^{\mathbb R^n}$ is given by
	\[
	\dist_{I}(M,N)\coloneqq
	\inf\{\eps> 0 \mid \exists \text{ an $\eps$-interleaving between $A$ and $B$}\}
	\]
	with the convention $\inf\emptyset = \infty$.
\end{definition}

The interleaving distance is strongly related to the shift amplitude $\shift$ of \cref{def:shift_amplitude} along the vector $(1,\dots,1)$. 
We sketch why this is the case: the natural map $\eta_{\varepsilon} \colon M \to M_{+\varepsilon}$ has the property that $\shift(\ker(\eta_{\varepsilon})), \shift( \coker (\eta_{\varepsilon})) \leq \varepsilon$. 
It follows that $\dist_{\shift}(M,M_{\varepsilon}) \leq 2\varepsilon$. 
Another immediate consequence of the defining equations for an $\varepsilon$-interleaving $f,g$ as above is that $\ker(f)\subset \ker(\eta_{2\varepsilon,M})$ and that $\coker(f)$ is a quotient of $\coker(\eta_{2\varepsilon,N})$, and thus $\shift( \ker(f)), \shift( \coker(f))  \leq 2 \varepsilon$. 
Hence, it follows from the triangle inequality that $\dist_{\shift}(M,N) \leq 6\dist_{I}(M,N)$. 
The converse inequality even holds without an additional factor, as demonstrated in \cite[Prop. 12.2]{oliver}.
In total, the following relationship between interleaving and shift amplitude distance holds. 

\begin{proposition}\label{prop:interleaving_shift_path_metric}\cite[Prop. 12.2]{oliver}
    Let $M,N \in \Vect^{\mathbb R^n}$, then we have
    \[
    \frac{1}{6}\dist_{\shift}(M,N) \leq \dist_{I}(M,N) \leq \dist_{\shift }(M,N) \, .
    \]
\end{proposition}

\begin{remark}
    The inequality of \cref{prop:interleaving_shift_path_metric} still holds in any scenario where the definition of the interleaving distances makes sense, i.e. on the categories $\permodx$, whenever $\mathfrak X$ is an encoding structure which is closed under shifts along $\vv{\varepsilon}$ for $\varepsilon >0$.
\end{remark}
Up to equivalence of metrics, we may thus study interleaving-style distances in the framework of amplitudes. 

\subsection{Wasserstein distances}

Another distance frequently used in persistence theory is the Wasserstein distance, which can be described in terms of matchings but also admits an algebraic interpretation. 
As shown in \cite{skraba_turner}, the Wasserstein distance can be reinterpreted as an $f$-span metric with $f$ given by $\pnorm{-}$ and $\alpha=\rho_p$.

\begin{definition}(\cite[Def. 7.7]{skraba_turner})\label{def_alg_wass}
Let $M$ and $N$ be two objects in $\permod[]$, and $p\in [1,\infty]$.
The algebraic $p$-Wasserstein distance between $M$ and $N$ is given by
\[
W_{p}^{alg}(M, N)=\dist_{f_{\alpha}}(M,N)=
\inf\limits_{(C,\varphi,\psi)}
\pnorm{\left( 
\rho_p\left(\ker\varphi\right), \rho_p\left(\coker\varphi\right), \rho_p\left(\ker\psi\right), \rho_p\left(\coker\psi\right)\right)} 
\]
where the infimum is taken over the spans $M\overset{\varphi}{\longleftarrow} C \overset{\psi}{\longrightarrow} N$, for $C\in\ob\permod[]$ and $f$ is the cost function $\pnorm{-}$.
\end{definition}

In \cite[Thm 7.27 and 7.28]{skraba_turner}, the authors prove that \cref{def_alg_wass} is indeed an alternative definition (up to equivalence of metrics) for the Wasserstein distance of persistence modules which are pointwise finite-dimensional with bounded $p$-energy. 
Using the equivalences of norms in $\mathbb{R}^4$ and the fact that $W_{p}^{alg}=\dist_{f_{\rho_p}}$, where $f$ is given by the $p$-norm on $\mathbb{R}^4$, we obtain the following equivalence:

\begin{corollary}
\label{th:relation_wasserstein_path} 
For $p\in [1,\infty]$ and $M,N \in \permod[]$, the following equivalence of metrics holds:
\[
W_{p}^{alg}(M,N) \leq \dist_{\rho_{p}}(M,N) \leq
4^{1-\frac{1}{p}} W_{p}^{alg}(M,N)\, .
\]
In particular, for $p=1$, we obtain $W_{1}^{alg} = \dist_{\rho_{1}}$.
\end{corollary}

\begin{remark}
In \cite[Th. 5.16]{BSS18}, a version of the case $p=1$ of \cref{th:relation_wasserstein_path} is proven under slightly different assumptions.
\end{remark}

In \cite{BSS18} the authors were looking to generalize the definition of Wasserstein distances to the multiparameter setting. 
They made use of path distances whose weight is an amplitude, constructed by integrating the Hilbert function $q \mapsto \dim M_q$ of a persistence module. 
It turns out that essentially every amplitude that is additive with respect to short exact sequences can be constructed in this fashion by replacing the Lebesgue measure with a \textit{content}. 
We explain this in more detail in \cref{SS:upper bound}.

\begin{example}\cite{BSS18}\label{e:Hilb_ampl}
    Let $\mathfrak X$ be a connective encoding structure on $\mathbb R^n$. 
    Every upset in $\mathbb R^n$ is Lebesgue measurable\footnote{ 
    See \cite[Thm.\ 4.4]{graham2006influence} for the case of $[0,1]^n$. 
    The proof there easily generalizes to $\mathbb R^n$.}. 
    Let $\lambda$ be the restriction of the Lebesgue measure to $\mathfrak X$. 
    Then we define the ($p$-)Hilbert amplitude of a module $M$ in $\permod[n]$, for $1 \leq p < \infty $:
    \begin{equation*}
    \hilbp{M}{p} =  \left(\int (\dim_{M})^p d \mu\right)^{\frac{1}{p}}
    \end{equation*}
    Note that this expression is indeed well defined, since $\dim_{M} \colon \mathbb R^n \to \mathbb N$ is a finite sum of indicator functions of Lebesgue measurable sets. 
    That is an amplitude is immediate from the elementary properties of $p$-norms and the behavior of Hilbert function under short exact sequences.
    If $p=1$, we simply write $\hilb{-}$.
    If $n=1$ and $p=1$, this coincides with $\rho_1$ (\cref{example_p-norms}). In particular, by the results of \cite{skraba_turner}, this means that $\hilbp{M}{p}$ is a generalization of the $1$-Wasserstein distance to the multiparameter scenario.
\end{example}
    
\cref{prop:exact_ampl_con} allows for a simple proof of the following stability result between interleaving distance and the measure distances defined in \cite{BSS18}. 

\begin{proposition}\label{thm:hilbert_bounds_inter}
Let $\mathfrak X$ be a connective encoding structure on $\mathbb R^n$ which is closed under shifts.
Let $\dist_I$ be the interleaving distance on $\permodx$ with respect to the vector $v=(1,\dots,1) \in \mathbb{R}^n$, $\dist_{\shift}$ the distance coming from the shift amplitude $\shift$ with respect to $v$ and $1 \leq p < \infty$.
Then we have
\[
\dist_I(M,N) \leq \dist_{\shift}(M, N) \leq 4 \left(\dist_{\hilbp{-}{p}}(M,N)\right)^{\frac{p}{n}}
\]
for all $M,N \in \ob \permod[n]$.
\end{proposition}

\begin{proof}
The first inequality is \cite[Prop. 12.2]{oliver}. 
We apply \cref{prop:exact_ampl_con}, to $g = \frac{p}{n}$ in Proposition \ref{prop:exact_ampl_con} taking $F$ to be the identity functor, $\alpha = \hilbp{-}{p}$, and $\alpha' = \shift$.
So, let $M \in \permod$ and, for any $u \in \mathbb R^n$ and $m\in M_u$, consider the unique morphism $\Field[u,\infty)\to M$ such that  $1 \mapsto m$. 
The image of this morphism is of the form $\Field [I]$ where $I= [u,\infty) \cap D$ for $D$ a downset in $\mathbb{R}^n$.  
By monotonicity of $\hilbp{-}{p}$, we have that $\hilbp{M}{p} \geq \lambda(S)^{\frac{1}{p}}$ where $\lambda$ is the Lebesgue measure and $S$ is the support of $M$. 
Note that if $x^{\eps v}m \neq 0$ then $S$ contains at least the cube with diagonal $\varepsilon$ with lower left corner at $u$. 
The latter has volume $\eps^n$. 
Hence, 
\[
\hilbp{M}{p} \geq \lambda(S)^{\frac{1}{p}} \geq \sup\{\eps \mid x^{\eps v}n \neq 0 \textrm{ for some }n \in M_w, w\in \mathbb R^n\}^{\frac{n}{p}} = \shift(M)^{\frac{n}{p}} .
\]
Applying $(-)^{\frac{p}{n}}$ concludes the proof.
\end{proof}

\subsection{The \texorpdfstring{$k$-}-longest interval distances}
In practice, often one only expects a certain number of intervals to carry relevant data. 
However, this number may vary in different situations, and one needs not only to be able to choose how many intervals to consider but also how to relate these different choices. 
We now discuss how to tackle these issues.

\begin{definition}
Let $k \in \N$. 
Given $M$ in $\ob\permod[]$, we write its interval decomposition $\oplus_{i=1}^n \mathbb{I}\langle b_i, b_i + \ell_i \rangle$, where each interval can be closed or open at endpoints, and assume (after reordering as necessary) that $\ell_1 \geq \ell_2 \geq \dots \geq \ell_n$.
Then the \define{$k$-longest-intervals amplitude} is defined as follows:
\[
\LI{k}\left(M\right) \coloneqq \sum_{i=1}^k \ell_i \, .
\]
\end{definition}

\begin{remark}
The $k$-longest-intervals amplitude is discussed as a key example of a ``tropical coordinate'' in \cite{K19_Tropical}.
Other tropical coordinates do not typically give amplitudes on $\permod[]$.
\end{remark}

\begin{proposition} \label{prop:li_amplitude}
For any $k \in \N$, the function $\LI{k}$ is an amplitude on $\permod[]$.
\end{proposition}

The proof relies on the matching theorem of \cite[Theorem 4.2]{BauerLesnick14}, which we recall for convenience.
For a persistence module $M$ in $\permod[]$, we denote by $\mathcal D_{M} = \{ \langle b^M_i, d^M_i \rangle \mid i=1,\dots,n\}$ the multiset of its barcode. 
Suppose we are given a short exact sequence of interval-decomposable persistence modules $0 \to A \to B \to C \to 0$. 
Then, by possibly adding empty intervals to the barcodes, we can find a common indexing of $\mathcal{D}_A$, $\mathcal{D}_B$, and $\mathcal{D}_C$ such that for all indices $i$, we have $d^A_i = d^B_i$, $b^B_i =b^C_i$, $b^A_i\geq b^B_{i}$, and $d^B_i \geq d^C_{i}$.
Graphically, for every bar in $B$, we can truncate the bar on the left to obtain a bar in $A$, and truncate the original bar on the right to obtain a bar in $C$, and every bar in $A$ and $C$ arises in this way.

\begin{proof}[Proof of \cref{prop:li_amplitude}]
Clearly $\LI{k}(0)=0$. 
Further, note that the monotonicity of $\LI{k}$ is an immediate consequence of the preceding discussion, as every bar in a submodule or quotient of a module $B$ corresponds injectively to a (potentially) longer bar in $B$.

We are left with showing subadditivity. 
As ephemeral modules (i.e. modules that are supported on points) do not contribute to $\LI{k}$, we assume without loss of generality all intervals involved to be left closed right open, to simplify the proof.

Given a short exact sequence $0 \to A \to B \to C \to 0$ of interval decomposable modules, let us first assume that $B$ has at most $k$ interval components. 
Then, by the discussion preceding this proof, so do $A$ and $C$.
Therefore, for $M = A$, $B$, or $C$, we have that $\LI{k}(M)$ is given by the sum of the lengths of all intervals.
In particular $\LI{k}(B) = \LI{k}(A) + \LI{k}(C)$.

For the general case, let $B'$ denote the submodule of $B$ given by the $k$ longest bars (in case there are multiple such modules, choose any). 
We then have a commutative diagram with vertical monomorphisms and exact rows
\[
\begin{tikzcd}
 0 \arrow[r] & \arrow[r] A' \arrow[d, hook] & \arrow[r]B' \arrow[d, hook] & \arrow[r] C' \arrow[d, hook]& 0 \\
0 \arrow[r]& A  \arrow[r] & B \arrow[r] & C \arrow[r] & 0
\end{tikzcd}
\]
where $A'$ is the intersection of $B'$ with the kernel of $B \to C$, and $C'$ is the image of $B'$ in $C$. 
By the special case above, together with monotonicity of $\LI{k}$, we obtain the claimed inequality:
\[\LI{k}(B) = \LI{k}(B') = \LI{k}(A') + \LI{k}(C') \leq \LI{k}(A) + \LI{k}(C) \, . \qedhere
\]
\end{proof}

We now show that we can always compare $k$-longest-interval amplitudes for different $k$ values.

\begin{lemma}\label{lem:id_pres_ampl_trop}
Let $k, \ell \in \mathbb{N}$ and let $\id_{\permod[]}\colon (\permod[],\LI{k})\to (\permod[],\LI{\ell})$ be given by the identity on objects.
\begin{enumerate}
    \item If $k \geq \ell$, then $\id$ is amplitude-bounding with constant $K=1$.
    \item If $1 \leq k \leq \ell$, then $\id$ is amplitude-bounding with constant $K = \ell /k$.
\end{enumerate}
\end{lemma}

\begin{proof}
1. This is immediate.

2. For a family of real numbers $(x_i)_{1 \leq i \leq N}$, we have the inequality
\[
\max_{S \subset \{1,\dots,N\}, |S| = \ell} \left( \frac{\sum_{i \in S} x_i}{\ell} \right) \leq \max_{S' \subset \{1,\dots,N\}, |S'| = k} \left( \frac{\sum_{i \in S'} x_i}{k} \right),
\]
Given $M$ in $\ob\permod[]$, let $(x_i)_{1 \leq i \leq N}$ be the lengths of the intervals in the barcode of $M$ (in arbitrary order).
Then 
\[
\LI{k}(A) = \max_{S \subset \{1,\dots,N\}, |S| = k} \sum_{i \in S} x_i,
\]
and similarly for $\LI{\ell}(A)$.
The stated inequality then implies the result.
\end{proof}

\begin{corollary}\label{cor:id_stable_k_longest}
Let $k,l \geq 1$ and let $\id_{\permod[]}\colon (\permod[],\LI{k})\to (\permod[],\LI{\ell})$ be given by the identity.
Then 
\[
\dist_{\LI{\ell}}(M,N)\leq \max \left\{1, \frac{\ell}{k} \right\} \dist_{\LI{k}}(M,N)
\]
for any $M,N \in\ob\permod[]$.
\end{corollary}

\begin{proof}
This follows from \cref{lem:id_pres_ampl_trop} and \cref{prop:Weaker_Ex_Cond}.
\end{proof}
Finally, let us discuss the relationship between Wasserstein distances and span metrics arising from the $k$-longest intervals amplitude.
\cref{cor:id_stable_k_longest} allows for a different perspective on the distances $\dist_{\LI{k}}$: they interpolate (in a monotone manner) between the interleaving and $1$-Wasserstein distances. 
Intuitively, it is quite clear: the interleaving distance is equal to the bottleneck distance \cite{Bubenik2014,isometry1,Lesnick2015}, i.e., it is measuring the cost of the longest interval, while the $1$-Wasserstein is adding the costs of all intervals. 
When $k$ increases, $\dist_{\LI{k}}$ considers the cost of more and more intervals, thus ``getting closer'' to the $1$-Wasserstein.  
Precisely, we have
\[ 
\dist_{I}(M,N) \leq \dist_{\LI{k}}(M,N) \xrightarrow{k \to \infty} W^1(M,N)\, .
\]
where the first inequality follows by \cite[Prop. 12.2]{oliver} and the limit is given by Corollary \ref{th:relation_wasserstein_path} since, for any $M,N \in \permod[]$, $\dist_{\LI{k}}(M,N) = W^1(M,N)$ for $k$ sufficiently large.

\section{Classifications of amplitudes in TDA}\label{S_classification_amplitudes}

In this section, we show how every amplitude on staircase persistence modules is bounded above and below by two classes of functions, and we show that these two classes are in bijection with two classes of amplitudes. 
To do so, we study amplitudes with respect to their values on interval modules. 
We then obtain a classification result for the classes of additive amplitudes, which are amplitudes that are additive on short exact sequences), as well as what we call ``quasi-simple'' amplitudes, which are amplitudes that behave like a maximum on direct sums.
Our classification of quasi-simple amplitudes provides a generalization of \cite[Thm. 9.6]{oliver}. 
\medskip

Throughout this section, we assume that $\Pos$ is a fixed poset and $\mathfrak X$ is a connective encoding structure on $\Pos$.

\begin{definition}\label{D:funct interv}
Given an amplitude $\alpha$ on $\permodx$, the \define{function on intervals associated to $\alpha$} is given by
\begin{align*}
    \AmpMeas \colon \AlgInt &\to [0,\infty] \\
    I &\mapsto \alpha(\Field[I]) \, .
\end{align*}
\end{definition}

We first note some properties of $\AmpMeas$ which may be immediately verified from the defining properties of an amplitude and \cref{cor:intervals_vs_modules}.

\begin{lemma}\label{L:prop funct interv}
    Let $\alpha$ be an amplitude on $\permodx$. 
    Then, for all $I,I_0,I_1 \in \AlgInt$ the following hold:
    \begin{enumerate}
        \item $\AmpMeas(\emptyset) = 0$ ;
        \item If $I_0 \subset I_1$ then 
        $\AmpMeas(I_0) \leq \AmpMeas(I_1)$ ;
        \item If $I = I_0\cup I_1$, where, for each $i= 0,1$, $I_i \subsetl I$ or $I_i \subsetr I$, then 
        $ \AmpMeas ( I ) \leq \AmpMeas (I_0) + \AmpMeas (I_1)$.
    \end{enumerate}
\end{lemma}

The function on intervals provides lower and upper bounds for the amplitude it is associated with:

\begin{proposition}\label{prop:upper_and_lower_bound_for_amp}
Let $M \in \permodx$ and $\alpha$ an amplitude on $\permodx$. 
Consider any $\mathfrak X$-encoding  of $M$, given by   $e \colon \Pos \to \Pos'$ and $M' \in \permodq{\Pos'}$. We have  the following inequalities:
\begin{equation}\label{intervals_bound_amplitude}
    \max_{p \leq q , M'_p \to M'_q \neq 0} \AmpMeas( e^{-1}[p,q]) \labelrel\leq{ineq_simple}  \alpha(M) \labelrel\leq{ineq_additive} \sum_{q \in \Pos'} \dim (M'_q) \AmpMeas(e^{-1}(q)) \, .
\end{equation}
\end{proposition}

\begin{proof}
We begin by proving inequality~\eqref{ineq_simple}.
Let $p\leq q$. 
If $M'_p \to M'_q$ is nonzero, then there exists a morphism $\phi \colon \Field[p, \infty) \to M'$, such that $\phi(t)$ is nonzero for all $t\in [p,q]$. 
Hence, the image of $\phi$ is an interval module $\Field[I]$, for some $I \subset [p,\infty)$, such that $[p,q] \subseteq I$. 
It follows that $\Field[p \leq q]$ is a quotient of $\Field[I]$, and we obtain the left-hand-side diagram in (\refeq{D:first ineq}). 
By applying $e^*$, we obtain the right-hand-side diagram in (\refeq{D:first ineq}). 

\begin{equation}\label{D:first ineq}
\begin{tikzpicture}[baseline=(current  bounding  box.center)]
\node (1) at (0,0) {$\Field[I]$};
\node (2) at (1.5,0) {$M'$};
\node (3) at (0,-1.5) {$\Field[p\leq q]$};
\draw[->>] (1)--(3);
\draw[->, right hook-latex] (1)--(2);
\node (1) at (5,0) {$e^\ast\left(\Field[I]\right)$};
\node (2) at (7.2,0) {$e^\ast\left(M'\right)$};
\node (3) at (5,-1.5) {$e^\ast\left(\Field[p\leq q]\right)$};
\draw[->>] (1)--(3);
\draw[->, right hook-latex] (1)--(2);
\end{tikzpicture}
\end{equation}

Since $e^*M' \cong M$,  inequality~\eqref{ineq_simple} follows by monotonicity of amplitudes.

We now prove inequality~\eqref{ineq_additive}. 
We proceed by induction over the number $n$ of elements $q \in \Pos'$ for which $M'_q \neq 0$. 
If $n=0$, the inequality is trivially satisfied. 
If $n=1$, and $q \in \Pos'$ is such that $M'_q \neq 0$ is of dimension $k$, then
\[
M \cong e^*(M') \cong e^* (\Field[q])^k \cong \Field[e^{-1}(q)]^k. 
\] 
It follows from the subadditivity of amplitudes that $\alpha(M) = \alpha(\Field[e^{-1}(q)]^k) \leq k \alpha(\Field[e^{-1}(q)])$. 
We now assume the claim to be true for $n$, and show that then it holds for $n+1$.
Take $q \in \max\left(\{p\in\Pos' \ \vert \ M'_q \neq 0\}\right)$. 
Denote by $M'' \in \permodq{\Pos'}$ the persistence module with $M''_p = M'_q$ if $p=q$ and $M''_p = 0$ otherwise. 
By maximality, there is a short exact sequence
\begin{equation}\label{ses_inequalities_classification}
0 \to M'' \to M' \to M' / M'' \to 0 \, .
\end{equation}
By construction, $\vert\{p\in\Pos'\ \vert \ \left(M'/M''\right)_p\neq 0\}\vert= n$, since $M'/M''$ and $M$ differ only for their evaluation at $q$.
In particular, the inductive assumption holds for to $e^*(M'/M'')$.
Applying $e^*$ to (\ref{ses_inequalities_classification}), we obtain a short exact sequence 
\[
0 \to e^*(M'') \to M \to e^* \left(M' / M''\right) \to 0. 
\]
From it, by subadditivity of amplitudes and inductive assumption, we obtain the claim:
\begin{align*}
    \alpha(M) & \leq \alpha(e^\ast(M'')) + \alpha(e^\ast \left(M' / M''\right)) \\ 
    & = \sum_{q' \in \Pos'} \dim (M''_{q'}) \AmpMeas(e^{-1}(q')) + \sum_{q' \in \Pos'} \dim (\left(M' / M''\right)_q) \AmpMeas(e^{-1}(q')) \\
    &= \sum_{q' \in \Pos'} \dim (M''_{q'} + \left(M' / M''\right)_{q'})  \AmpMeas(e^{-1}(q')) 
    = \sum_{q' \in \Pos'} \dim (M'_{q'})  \AmpMeas(e^{-1}(q')) \, .\qedhere
\end{align*}
\end{proof}

Next we aim to characterize the amplitudes for which one of the two bounds in Equation \refeq{intervals_bound_amplitude} is attained (for all $M$): in Subsection \ref{SS:lower bound} we classify the class of amplitudes that realizes the lower bound \eqref{ineq_simple}, while in Subsection \ref{SS:upper bound} we classify the one that realizes the upper bound \eqref{ineq_additive}.

\subsection{Quasi-simple amplitudes}\label{SS:lower bound}

The goal of this subsection is to show that the amplitudes for which inequality~\eqref{ineq_simple} is an equality, for any encoding, behave like a maximum on direct sums of persistence modules (\cref{thm:classification_of_quasi_simple_amp}). 
We call this type of amplitudes \textit{quasi-simple}.
To reach this goal we first describe some characterizing properties of quasi-simple amplitudes (\cref{prop:char_quasi_simple_amp} and \cref{def_quasisimple}).
Then we define a class of functions over intervals that behave essentially like taking the diameter of an interval in the direction of the poset and are therefore called \textit{directed diameters} (\cref{def_directeddiam}). 
In this section, we prove \cref{thm:classification_of_quasi_simple_amp}, which shows that amplitudes which behave like a maximum under direct sums are in bijection with directed diameters. 
\begin{proposition}\label{prop:char_quasi_simple_amp}
Let $\alpha$ be an amplitude on $\permodx$. 
Then the following conditions are equivalent:
\begin{enumerate}
    \item For all $M,N \in \permodx$, $\alpha( M \oplus N) = \max \{\alpha(M), \alpha(N)\}$.
    \item Let $I_i \in \AlgInt$. For all $M \in \permodx$, the following holds: If there exists an epimorphism $\bigoplus_{i=1,\dots, n} \Field[I_i] \twoheadrightarrow M$ such that the resulting morphism $\Field[I_i]\to M$  is  a monomorphism for every $i$, then \[ \alpha(M) = \max \{\AmpMeas(I_i) \mid i= 1,\dots,n\} \, . \]
    \item Let $I_i \in \AlgInt$. For all $M \in \permodx$, the following holds: If there exists a monomorphism $M \hookrightarrow \bigoplus_{i=1,\dots, n} \Field[I_i]$  such that the resulting morphism $ M\to \Field[I_i]$  is  an epimorphism for every $i$, then \[\alpha(M)  = \max \{\AmpMeas(I_i) \mid i= 1,\dots,n\} \, .\]
    \item For every $M \in \permodx$ and every encoding of $M$ with $e \colon \Pos \to \Pos'$ and $M' \in \permodq{\Pos'}$, we have: \[\alpha(M) = \max_{p \leq q , M'_p \to M'_q \neq 0} \AmpMeas( e^{-1}([p,q])\, .\] 
    \end{enumerate}
\end{proposition}

\begin{proof}
The implications (1.$\implies$2.) and (1.$\implies$3.) follow directly from the amplitudes' axioms. 
We now show (2.$\implies$1.). 
For $M, N \in \permodx$ choose upsets $U_i,V_j \in \mathfrak X$, with $i=1, \dots, n$ and $j=1,\dots,m$, and epimorphisms $\oplus \phi_i\colon \oplus \Field[U_i] \twoheadrightarrow M$ and $\oplus \psi_i\colon \oplus \Field[V_i] \twoheadrightarrow N$, see \cite[Theorem 6.12]{Miller-hom}.
Let $I_i \subseteq U_i$ and $J_j \subseteq V_j$ be the subsets given by the points in $\Pos$ where $\phi_i$ and $\psi_j$ are nonzero, respectively. 
We then have induced natural isomorphisms $\Field[I_i] \cong \textnormal{im} (\phi_i)$ and $\Field[J_j] \cong \textnormal{im} (\psi_j)$.
In particular, we obtain induced epimorphisms $\oplus \Field[I_i] \twoheadrightarrow M$ and $\oplus \Field[J_j] \twoheadrightarrow N$, which are componentwise monomorphisms,
and consequently also an epimorphism
\[
\oplus \Field[I_i] \bigoplus \oplus \Field[J_j] \twoheadrightarrow M \oplus N,
\]
which is componentwise a monomorphism.
Then, applying the assumption twice, we have
\begin{align*}
\alpha(M \oplus N) &= \max_{I\in \{I_1,\dots, I_n,J_1,\dots , J_m\}}\{\AmpMeas(I)\}
    =\max \left\{ \max_{i=1,\dots,n} \{\AmpMeas(I_i)\},  \max_{i=j,\dots,m} \left\{\AmpMeas(J_i) \right\} \right\}\\
    &= \max\{\alpha(M),\alpha(N)\} \, .
\end{align*}

The proof of (3.$\implies$1.) follows a similar argument by duality. 

To see (2.$\implies$4.), choose some encoding $e\colon \Pos \to \Pos'$, $M' \in \permodq{\Pos'}$ of $M \in \permodx$. 
We have an epimorphism $\phi\colon \oplus_{p \in \Pos'} (\Field[p,\infty))^{\dim(M'_p)} \twoheadrightarrow M'$, which at some fixed $p \in \Pos'$ is an isomorphism when restricted to the $p$-component.
Next, we note that by \cref{cor:intervals_vs_modules}, the image of such a monomorphism $\phi_p$ is isomorphic to an interval module $\Field[I_{p}]$ where $I_{p} \subseteq \Pos'$ is an interval having a unique minimal element $p$.
Thus, by taking the direct sum of these interval modules, we have an epimorphism $\oplus (\Field[I_{p}]) \twoheadrightarrow M'$, which is componentwise a monomorphism. 
Diagrammatically, we illustrate the argument below. 
\[
\begin{tikzpicture}
\node (1) at (-1,0) {$\bigoplus_{p \in \Pos'} (\Field[p,\infty))^{\dim(M'_p)}$};
\node (2) at (3,0) {$M'$};
\node (3) at (-1,-2) {$\Field[p,\infty)$};
\node (4) at (3,-2) {$\mathrm{im} \phi_p$};
\node (5) at (6,0) {$\bigoplus \mathrm{im} \phi_p\cong \bigoplus (\Field[I_{p}]) $};
\draw[->>] (1) edge node[above] {$\phi$} (2);
\draw[->>] (5) -- (2);
\draw[->, right hook-latex] (3) -- (1);
\draw[->, right hook-latex] (4) -- (2);
\draw[->, right hook-latex] (4) -- (5);
\draw[->] (3) edge node[left] {$\phi_p\;\;\;$} (2);
\draw[ ->>] (3) -- (4);
\end{tikzpicture}
\]

Since $e^\ast$ is exact, we obtain an epimorphism $\oplus (\Field[e^{-1}(I_{p})]) \twoheadrightarrow M$, which is also a monomorphism componentwise.
It hence follows by the assumption that there exists an $s$ such that
\[
\alpha(M) = \AmpMeas( e^{-1}(I_s)) \, .
\] 
Let $p$ be the minimal element of $I_s$ and $q_0, \dots, q_n$ the finite set of maximal elements of $I_s$. 
Then $I_s = \bigcup_{i = 0, \dots, n} [p, q_i]$.
Thus, we have the canonical monomorphism $\Field[I_s] \hookrightarrow \oplus_{i = 1, \dots, n}\Field[p,q_i]$, which is componentwise an epimorphism.
This induces a monomorphism ${\Field[e^{-1}(I_0)] \hookrightarrow \oplus_{i = 1, \dots, n}\Field[e^{-1}[p,q_i]]}$
which is componentwise an epimorphism. 
Since we already showed that (2.$\iff$3.), from the hypothesis we have
\[
\alpha(M) = \AmpMeas( e^{-1}(I_s)) = \max_{i = 1, \dots, n} \AmpMeas (e^{-1}[p,q_i]). 
\]
Since $[p,q_i] \subseteq I_s$, for each $i=1,\dots,n$ it follows that $M'_{p} \to M'_{q_i}$ is nonzero. 
Hence we obtain
\[
\alpha(M) = \max_{p \leq q , M'_p \to M'_q \neq 0} \AmpMeas( e^{-1}([p,q]]).
\]
    
Finally, we show (4.$\implies$1.). 
Let $M,N \in \permodx$ and let $e\colon \Pos \to \Pos'$ be a common encoding map of $M$ and $N$, with encoding modules $M'$ and $N'$. We thus have  $e^\ast(M' \oplus N') \cong M \oplus N$. 
We note that $(M' \oplus N')_{p} \to (M' \oplus N')_{q}$ is nonzero if and only if one of the two components is nonzero. 
Thus, 
\[ \bigg\{p \leq q \ \vert \ (M'\oplus N')_p \to (M'\oplus N')_q \neq 0\bigg\}= \bigg\{p \leq q \ \vert \ M'_p \to M'_q \neq 0\bigg\} \bigcup \bigg\{p \leq q \ \vert \ N'_p \to N'_q \neq 0\bigg\}\,\]
and the claim follows.
\end{proof}

\begin{definition}\label{def_quasisimple}
    An amplitude $\alpha$ on $\permodx$ is \define{quasi-simple} if it satisfies any of the equivalent properties of \cref{prop:char_quasi_simple_amp}.
\end{definition}

For example, the shift amplitude (\cref{def:shift_amplitude}) is quasi-simple amplitude. 

\begin{remark}\label{R:simple amplitudes} 
We note that quasi-simple amplitudes generalize the simple noise systems defined in \cite{SC+19}. 
Under the equivalence of noise systems and amplitudes (see Proposition \ref{prop:equ_of_cat}), we have that noise systems can be thought of as corresponding to a specific kind of quasi-simple amplitudes. 
Namely, if one replaces finitely encoded by finitely presented persistence modules, simple noise systems  correspond to quasi-simple amplitudes that additionally satisfy the following two properties: (1) for every persistence module $M$  and every $\varepsilon \geq 0$, there exists a minimal subobject $M'$ with the property that $\alpha(M/M') \leq \varepsilon$; and (2) the subobject $M'$ has rank less or equal to $M$.
\end{remark}

As an immediate corollary of the characterization of quasi-simple amplitudes in \cref{prop:char_quasi_simple_amp} we obtain:

\begin{lemma}\label{lem:func_of_quasi_simple}
    If $\alpha$ is quasi-simple then for all  $I_0,I_1,I \in \mathfrak X$, where  $I_0$ and $I_1$ are both upsets or both downsets of $I$,  we have
    \[
    \AmpMeas ( I ) = \max \{ \AmpMeas (I_0),\AmpMeas (I_1) \} \, .
    \]
\end{lemma}

\begin{definition}\label{def_directeddiam}
A function $\dirdiam \colon \AlgInt \to [0, \infty]$ is a \define{directed diameter}, if for all $I,I',I_0,I_1 \in \AlgInt$ it fulfills the following properties:
    \begin{enumerate}
    \item $\dirdiam(\emptyset) = 0$;
    \item If $I'\subset I$, then $\dirdiam(I') \leq \dirdiam(I)$ ;
    \item If $I= I_0 \cup I_1$, $I_0 \subsetl I$, and $I_1 \subsetr I$, then $\dirdiam ( I ) \leq \dirdiam (I_0) + \dirdiam (I_1)$;
    \item If $I= I_0 \cup I_1$ and $I_0,I_1 \subsetr I$ or $I_0,I_1 \subsetl I$, then $ \dirdiam ( I ) = \max \{ \dirdiam(I_0), \dirdiam(I_1) \}$.
\end{enumerate}
\end{definition}

\begin{example}
Let $\mathfrak C$ be the staircase encoding structure and $C$ a cone in $\mathbb R^n$, and $\shift$ the shift amplitude (\cref{def:shift_amplitude}).
Then the function on interval associated to $\shift$, $\mu_\shift$, is a directed diameter on $\mathfrak C^{\textnormal{int}}$. 
Explicitly 
\begin{equation*}
    I \mapsto \sup \{ \pnorm[]{c} \mid c\in C \textnormal{ and } t + c \in I \textnormal{ for some } t\in I \} 
\end{equation*}
or equivalently
\begin{equation*}
    I \mapsto \inf \{ \pnorm[]{c} \mid c\in C \textnormal{ and } t + c \notin I \textnormal{ for all } t\in I \}.
\end{equation*}
The value $\mu_{\shift}(I)$ measures the diameter of $I$ in the direction of the cone $C$.
\end{example}

Before proving that the above example is not just a coincidence, i.e. that the function on interval associated to any quasi-simple amplitude is a directed diameter, we first note the following symmetry property of directed diameters:

\begin{proposition}\label{prop:dir_diam_amp}
Let $\dirdiam$ be a directed diameter on $\AlgInt$. 
Then, for every $M \in \permodx$, we have the following equalities
\begin{align*}
    \sup \{\dirdiam(I) \mid I \in \AlgInt, \Field[I] \hookrightarrow M \} &= \sup \{\dirdiam(I) \mid I \in \AlgInt, M \twoheadrightarrow \Field[I] \} \\
    &= \sup \{\dirdiam(I) \mid I \in \AlgInt, \textnormal{ $\Field[I]$ is a subquotient of $M$} \}
\end{align*}
\end{proposition}
\begin{proof}
It is straightforward to see that the first two expressions are lesser than or equal to the final one. 
By dualizing, it thus suffices to show that 
\[
\sup \{\dirdiam(I) \mid \textnormal{$I$ is a subquotient of $M$} \} \leq \sup \{\dirdiam(I) \mid \Field[I] \hookrightarrow M \}.
\]
Suppose we are given a subquotient zigzag $M \hookleftarrow N \twoheadrightarrow \Field[I]$. 
As in the proof of \cref{prop:char_quasi_simple_amp} we may choose interval modules $I_1, \dots, I_n \in \AlgInt$ together with an epimorphism
\[
\bigoplus_{i=1,\dots,n}\Field[I_i] \twoheadrightarrow N
\]
which is componentwise a monomorphism. 
Consequently, we also have an epimorphism
\[
\bigoplus_{i=1,\dots,n}\Field[I_i] \twoheadrightarrow \Field[I].
\]
By replacing $I_i$ with the interval corresponding to the image of $\Field[I_i] \to \Field[I]$, denoted $\hat I_i$, we obtain an epimorphism
\[
\bigoplus_{i=1,\dots,n} \Field[\hat I_i] \twoheadrightarrow \Field[I],
\]
which is componentwise a monomorphism. 
It follows by \cref{cor:intervals_vs_modules} that $I = \bigcup_{i=1,\dots,n} \hat I_i$ and $\hat I_i \subsetr I$. Consequently, by assumption, $\dirdiam(I) = \dirdiam(\hat I_k)$, for some $k$. 
By construciton, there is an epimorphism $\Field[I_k] \twoheadrightarrow \Field[\hat I_k]$. 
Hence, again by \cref{cor:intervals_vs_modules}, it follows that $\hat I_k \subsetl I_k$ and consequently that $\dirdiam(I) = \dirdiam(\hat I_k) \leq \dirdiam(I_k)$. 
The composition of monomorphisms
\[
\Field[I_k] \hookrightarrow N \hookrightarrow M,
\]
show that $\Field[I_k]$ is a submodule of $M$, with $\dirdiam (I_k) \geq \dirdiam (I)$, as claimed.
\end{proof}

\begin{theorem}\label{thm:classification_of_quasi_simple_amp}
The assignment $\alpha \mapsto \AmpMeas$ induces a bijection between quasi-simple amplitudes on $\permodx$ and directed diameters on $\AlgInt$.
The inverse map is given by $\dirdiam \mapsto \alpha_{\dirdiam}$, 
where $\alpha_{\dirdiam}\left(M\right) \coloneqq \sup \{\dirdiam(I) \mid I \in \AlgInt, \ \Field[I] \hookrightarrow M \} \big \}$.
\end{theorem}

\begin{proof}
By \cref{lem:func_of_quasi_simple} the image of $\alpha \mapsto \mu_{\alpha}$ is a directed diameter. 

We now show that $\alpha_{\dirdiam}$ is an amplitude. 
Note that by \cref{prop:dir_diam_amp}, we may freely replace the monomorphism in the definition of $\alpha_{\dirdiam}$ by an epimorphism or subquotient.
From the assumption $\mu( \emptyset ) =0$, we immediately have $\alpha_{\dirdiam}(0) = 0$. 
Monotonicity under monomorphisms (epimorphisms) follows directly from the monomorphism (epimorphism) description of $\alpha_\mu$, respectively. 
Finally, to prove the subadditivity, consider the diagram
\[\begin{tikzcd}
 & & \Field[I] \arrow[d, hook, "i"] & &  \\
0 \arrow[r] & M' \arrow[r] & M\arrow[r, "g"] & M'' \arrow[r]& 0
\end{tikzcd}\, ,
\]
with the lower row exact. 
We can complete the diagram into a commutative diagram with exact rows as follows:
\[\begin{tikzcd}
    0 \arrow[r] & i^{-1}M' \arrow[r] \arrow[d, hook]& \Field[I] \arrow[d, hook, "i"] \arrow[r] & \mathrm{im} g \circ i \arrow[d, hook] \arrow[r]& 0  \\
    0 \arrow[r] & M'\arrow[r] & M\arrow[r, "g"] & M''\arrow[r] & 0
\end{tikzcd}\, .
\]
Since $i^{-1}M'$ and $ \mathrm{im} g \circ i$ are, respectively, a submodule and a quotient of $\Field[I]$, by \cref{cor:intervals_vs_modules}, they are of the form $\Field[I_0]$ with $I_0 \subsetr I$ and $\Field[I_1]$ with $I_1 \subsetl I$, respectively, and $I = I_0 \cup I_1$. 
It follows that $\dirdiam(I) \leq \sum \dirdiam(I_0) + \dirdiam(I_1)$, which shows subadditivity. 

Next, we show that $\alpha_{\dirdiam}$ is quasi-simple. 
It is enough to prove that, for every $M,N \in \permodx$,
\[\alpha_{\dirdiam}(M_0 \oplus M_1) \geq \alpha_{\dirdiam}(M_0) + \alpha_{\dirdiam}(M_1)\, .
\] 
Given a monomorphism $\Field[I] \hookrightarrow M_0 \oplus M_1$ with $I \in \AlgInt$, we consider the image of the two components of this morphism. 
Since they are given by quotient objects of $\Field[I]$ by \cref{cor:intervals_vs_modules}, they are of the form $I_0, I_1 \subsetl I$, for $I_0,I_1 \in \AlgInt$. Furthermore, by construction, we have a monomorphism $   \Field[I] \hookrightarrow \Field[I_0] \oplus \Field[I_1]$. It follows, again by \cref{cor:intervals_vs_modules}, that $I = I_0 \cup I_1$ and therefore that $\dirdiam (I) = \max \{  \mu (I_0), \mu( I_1) \}$. 
Without loss of generality, assume that the maximum is attained for $I_0$.  
Then, since by construction we have an embedding $\Field[I_0] \hookrightarrow M_0$, it follows that $\alpha_{\dirdiam}(M) \leq  \alpha_{\dirdiam}(M_0)$, i.e. $\alpha_{\dirdiam}$ is quasi-simple.    
    
Finally, let us verify that the two constructions are inverse to each other. 
By \cref{prop:char_quasi_simple_amp}, in particular from the second and third characterization, it follows that if we start with an amplitude, pass to a directed diameter, and then pass back to amplitudes, we re-obtain the original amplitude. 
Conversely, by \cref{cor:intervals_vs_modules}, every interval $I$ is always maximal among intervals $J$, with the property that there exists a monomorphism $\Field[J] \hookrightarrow \Field[I]$, it also follows that starting with a directed diameter, passing to an amplitude, and then taking the function on intervals is the identity. 
\end{proof}

As an immediate corollary of \cref{prop:dir_diam_amp} and we \cref{thm:classification_of_quasi_simple_amp} obtain:

\begin{corollary}
Let $\alpha$ be a quasi-simple amplitude on $\permodx$. 
Then the following equalities hold for any $M \in \permodx$:
\begin{align*}
\sup \{\alpha(\Field[I]) \mid I \in \AlgInt, \Field[I] \hookrightarrow M \} & = \sup \{\alpha(\Field[I]) \mid I \in \AlgInt, M \twoheadrightarrow \Field[I] \} 
\\ 
&= \sup \{\alpha(\Field[I]) \mid I \in \AlgInt, \textnormal{ $\Field[I]$ is a subquotient of $M$} \}\, .
\end{align*}
\end{corollary}

\subsection{Additive amplitudes}\label{SS:upper bound}

Next, we classify the amplitudes for which inequality~\eqref{ineq_additive} is strict. 
We show that these are the \textit{additive amplitudes} (see Definition \ref{D:amplitude}), and can be described in terms of integrating, in some sense, the Hilbert function of a persistence module\footnote{Such a characterization may be understood as a Riesz-(Markov)-Representation type theorem for additive amplitudes (see \cite[Thm. 2.14]{rudinRealandComp} for a reference of the latter).}.  
As a consequence, we have that additive amplitudes take the same values on persistence modules that have the same Hilbert function but potentially different structure morphisms.   

\begin{proposition}\label{prop:char_additive_amp}
Let $\alpha$ be an amplitude on $\permodx$. 
Then the following conditions are equivalent:
\begin{enumerate}
    \item $\alpha$ is additive.
    \item For $M \in \permodx$  and $M' \in \permodq{\Pos'}$, $e \colon \Pos \to \Pos'$ part of an $\mathfrak X$-encoding of $M$, we have
    \[
    \alpha(M) = \sum_{q \in \Pos'} \dim (M'_q) \AmpMeas(e^{-1}(q))\, .
    \]
\end{enumerate}
\end{proposition}

\begin{proof}
The implication (1.$\implies$2.) may be proved with an inductive argument,  analogously to the proof of \cref{prop:upper_and_lower_bound_for_amp}, inequality \eqref{ineq_additive}, if one replaces the appropriate inequalities by equalities.  
    
For (2.$\implies$1.), we need to show that, given a short exact sequence $S = 0 \to A \to B \to C \to 0$ in $\permodx$, $\alpha(B)=\alpha(A)+\alpha(C)$. 
By \cite[Prop. 3.6]{Lukas}, we can take a sequence $S' = 0 \to A' \to B' \to C' \to 0$ in $\permodq{\Pos'}$, such that by pulling back $S'$ along $e \colon  \Pos \to \Pos'$ we obtain $S$. 
Moreover, $e$ can be chosen such that $e^*$ is fully faithful, and thus $S'$ is also exact.
Therefore we have for each $q \in \Pos'$
\[
\dim B'_q = \dim A'_q + \dim C'_q \, . 
\]
Consequently, we obtain:
\begin{align*}
    \alpha(B) &= \sum_{q \in \Pos'} \dim (B'_q) \AmpMeas(e^{-1}(q))
    = \sum_{q \in \Pos'} \dim (A'_q) \AmpMeas(e^{-1}(q)) + \sum_{q \in \Pos'} \dim (C'_q) \AmpMeas(e^{-1}(q))\textbf{} \\
    &= \alpha(A) + \alpha(C) \, . \qedhere
\end{align*}
\end{proof}

A direct consequence of the above characterization of additive amplitudes is that the associated function on intervals is completely determined by its behavior on the connected components of each interval.

\begin{lemma}\label{lem:add_gives_cont}
Let $\alpha$ be an additive amplitude on $\permodx$. 
Then, for all $I, I_1, \dots, I_n \in \AlgInt$ with $I = \bigcup_{i=1}^{n} I_i$ and $I_i\cap I_j=\emptyset$ for all $i\neq j$, it holds that
\[
\AmpMeas(I) = \sum_{i= 1}^{n} \AmpMeas(I_i)\, .
\]
\end{lemma}

\begin{proof}
Choose a common encoding map $e \colon \Pos \to \Pos'$ of class $\mathfrak X$, as well as encoding modules $F$ and $F_i$, for $\Field[I]$ and $\Field[I_i]$ respectively. 
Then, using the definition of function on intervals, Lemma \ref{L:prop funct interv}, Proposition \ref{prop:char_additive_amp}, and \cref{rmk_encoding_disjoint_intervals}, we have
\begin{equation*}
    \AmpMeas(I) = \alpha(\Field[I]) = 
    \sum_{q \in e(I)} \AmpMeas(e^{-1}(q)) = \sum_{i= 1}^{n}\sum_{q \in e(I_i)} \AmpMeas(e^{-1}(q)) = \sum_{i= 1}^{n} \alpha(\Field[I_i]) = \sum_{i= 1}^{n}\AmpMeas(I_i) \, . \qedhere
\end{equation*}
\end{proof}

We now recall a classical notion from measure theory, which is also known as \emph{valuation} in integral geometry (see for example \cite{book_valuations}).

\begin{definition}
A \define{content} on an algebra $\mathfrak A$ is a function $\cont \colon \mathfrak A \to [0,\infty]$ 
such that:
\begin{enumerate}
    \item $\cont(\emptyset) = 0$;
    \item For all $S_1, S_2 \in \mathfrak A$ with $S_1 \cap S_2=\emptyset$, $\cont (S_1 \cup S_2) =  \cont(S_1) + \cont(S_2)$.
\end{enumerate}
By a \define{content on $\AlgInt$} we mean a function $\cont \colon \mathfrak X \to [0,\infty]$ which extends to a content on $\mathfrak X$.
\end{definition}

\begin{remark}
    Note that as $\mathfrak X$ is generated by $\AlgInt$ under disjoint unions, a function $\cont \colon \mathfrak X \to [0,\infty]$ extends to a content on $\mathfrak X$ if, and only if for all $I, I_1, \dots, I_n \in \AlgInt$, for some $n\geq 0$, with $I = \bigcup_{i=1}^{n} I_i$ and $I_i\cap I_j=\emptyset$ if $i\neq j$, it holds that $\cont( I ) =  \sum_{i = 1}^{n} \cont(I_i)$. 
    Furthermore, this extension is unique, so we really will not distinguish between contents on $\AlgInt$ and on $\mathfrak X$.     
\end{remark}

By the \cref{lem:add_gives_cont}, we have that, for every additive amplitude $\alpha$, its function on intervals defines a content on $\mathfrak X$. 
\medskip

Recall that a function $f \colon S \to \mathbb N$ is called measurable with respect to an algebra $\mathfrak A$ on $S$ if each of its fibers lies in $\mathfrak A$. 
Furthermore, measurable functions $f \colon S \to \mathbb N$ can be integrated with respect to a content $\cont$ on $\mathcal A$ via the formula $\int f d \cont = \sum_{n \in \mathbb N} f(n) \cont(f^{-1}(n))$.

\begin{remark}\label{ex:Hilbert_meas}
    If $M \in \permodx$, then its Hilbert function 
    \begin{align*}
       \dim_{M} \colon \Pos \to \mathbb N \\
        q \mapsto \dim M_q
    \end{align*}
    is $\mathfrak X$-measurable. 
Indeed, choose any encoding function $e \colon \Pos \to \Pos'$ for $M$. 
Then every fiber of $\dim_M$ is given by a finite union of fibers of $e$. 
As any such fiber is an interval of $\mathfrak X$, measurability follows. 
\end{remark}

We can then phrase the relationship between contents and additive amplitudes as follows.
\begin{theorem}\label{thm:classification_of_additive_amp}
The assignment $\alpha \mapsto \mu_\alpha$ induces a bijection between additive amplitudes of $\permodx$ and contents on $\mathfrak X$. 
The inverse map is given by 
\[\cont \mapsto \{ M \mapsto \int \dim_{M} d \cont \}\, .
\]
\end{theorem}

\begin{proof}
By \cref{lem:add_gives_cont}, we know that for any additive amplitude $\alpha$ its function on intervals $\mu_\alpha$ is a content on $\AlgInt$.

Conversely, we denote by $\cont$ any content on $\mathfrak X$.
We have that the integral $ M \mapsto \int \dim_{M} d \mu$ is well defined by \cref{ex:Hilbert_meas}.
It is immediate from the additivity of Hilbert functions under short exact sequences that this expression does indeed define an amplitude on $\permodx$. 

It remains to see that the two constructions are inverse to each other. 
Since $ \langle\AlgInt\rangle = \mathfrak X$, it is enough to show that given a content $\cont$ on $\mathfrak X$ we can recover the content on $\AlgInt$. 
Indeed, for any $I \in \AlgInt$ we have
\[
\int \dim_{\Field[I]} d \mu = \mu(I).
\]
On the other hand, if we begin with an amplitude $\alpha$, then for any $M \in \permodx$ and $e \colon \Pos \to \Pos'$ part of an $\mathfrak X$-encoding of $M$ using \cref{prop:char_additive_amp}, we have
\begin{align*}
    \alpha(M) & = \sum_{q \in \Pos'} \dim (M'_q) \AmpMeas(e^{-1}(q)) = \sum_{n \in \mathbb N} n \sum_{q \in \Pos', \dim (M'_q) = n} \AmpMeas(e^{-1}(q))  \\
    &= \sum_{n \in \mathbb N} n \AmpMeas(\dim_M^{-1}(n)) 
    = \int \dim_{M} d\mu_\alpha \, . \qedhere
\end{align*}
\end{proof}

\begin{remark} 
In a sense, \cref{thm:classification_of_additive_amp} is a (monoid and) persistence modules version of the classical sheaf function correspondence of \cite[Thm. 9.7.1]{SheavesOnManifolds}. 
Indeed, in the setting of constructible persistence modules in the sense of \cite{berkouk2022persistence}, which correspond to constructible sheaves fulfilling certain microsupport conditions (see \cite{Berkouk_2021}) similar results were proven in \cite{berkouk2022persistence}.
Hence, one could have also obtained \cref{thm:classification_of_additive_amp} by computing the Grothendieck monoid of $\permodx$. 
Recall that the Grothendieck monoid of an Abelian category $\mathcal A$ (first defined in \cite{berenstein-greenstein})
is a commutative monoid $\mathcal M_{\mathcal A}$, together with a map $\ob \mathcal{A} \to \mathcal M_{\mathcal{A}}$ which sends $0$ to $0$, is additive under short exact sequences, and is universal amongst such monoid valued maps. 
$\mathcal{M}_{\mathcal{A}}$ may explicitly be constructed by taking the free monoid on $\ob \mathcal{A}$ (with $0$ as neutral element), and quotienting out by $A + C \sim B$, whenever there is a short exact sequence $0 \to A \to B \to C \to 0$. 
It follows from the universal property of the Grothendieck monoid that additive amplitudes on $\mathcal A$ correspond to monoid morphisms $\mathcal M_{\mathcal A} \to [0,\infty]$. 

If $\A = \permodx$, the Grothendieck monoid is the monoid of bounded maps $f \colon \Pos \to \mathbb N$ which are $\mathfrak X$-measurable. 
Denote the latter monoid by $\mathcal{M}_{\mathfrak X}^{n}(\Pos, \mathbb N)$. 
The bijection is induced by the Hilbert function. 
Indeed, by the same inductive argument as in the proof of \cref{prop:upper_and_lower_bound_for_amp}, one shows that the map sending $M$ to its Hilbert function is injective. 
Conversely, to see that bounded $\mathfrak X$-measurable $\mathbb N$ valued function may be obtained as a Hilbert function, one can show that the constant regions of such a function may further be subdivided into intervals of $\AlgInt$. 
Since indicator functions of intervals lie in the image of the Hilbert function construction, we obtain surjectivity.
Finally, one may verify that the set of monoid morphisms 
\[
\Hom_{\textnormal{Grpd}}( \mathcal{M}_{\mathfrak X}^{n}(\Pos, \mathbb N), [0, \infty])
\]
is equivalently the set of contents on $\mathfrak X$.
\end{remark}

\section{Discriminativity of amplitude metrics in TDA}\label{S:discriminativity}

In Section \ref{S_path_span_metrics} we introduced span metrics associated with an amplitude on a category of persistence modules. 
As we have seen, the span metric is an extended pseudometric. 
In this section, we investigate under which conditions the span metric is an extended metric on the set of isomorphism classes of persistence modules. 
In other words, when are two modules with amplitude-distance $0$ isomorphic?
To answer this question, we start with a first assumption: we only allow for amplitudes that do not send non-zero objects to $0$.

\begin{definition}
An amplitude $\alpha$ is \define{strict} if $\alpha(A)=0 \implies A \cong 0$ for all $A \in \A$.
\end{definition}

\begin{example}
Many classical examples of amplitudes are not strict, even in the one-parameter setting. 
Consider for example the shift amplitude on $\permod[]$ induced by the directed diameter sending an interval to its length.
Its metric is equivalent to the interleaving distance of $1$-parameter persistence modules.
For interval modules of the shape $\Field[p,p]$, the shift amplitude takes the value $0$. More generally, it is $0$ for all ephemeral persistence modules (see \cite{chazal2014observable,Berkouk_2021}). 
We may restore strictness by quotienting out by the subcategory of ephemeral persistence modules. 
This is equivalent to instead restricting to the encoding structure $\cCubes$, i.e. only allowing for half-open intervals closed at the left and open at the right (compare with \cite{Berkouk_2021}). 
\end{example}

The question of discriminativity of the span metric can be phrased as a convergence question for spans.
Specifically, does the existence of spans $A \leftarrow C_{\eps} \to B$ with arbitrarily small cost allow one to conclude that there is a span $A \leftarrow C_{0} \to B$ with cost $0$? 
In general, it is difficult to answer if one does not require some type of tameness assumption of the objects in use.
In the finitely generated case, for example, the diagonal shift amplitudes are known to induce a metric on isomorphism classes (see \cite{Lesnick2015}). 
In the framework of constructible sheaves in \cite{petit2021property} it is shown that interleaving type distances on categories of appropriately constructible sheaves have this property.

In this section, we show such a result for general amplitude distances over persistence modules under certain topological boundary assumptions. 
We begin by setting up the framework:

\begin{definition}
We say that a subspace $\mathbb P \subset \mathbb R^n$ is \define{locally open above} if for all $x \in \mathbb P$ there exists a neighborhood $U \subset \mathbb R^n$, such that $(x + \mathbb R^n_{\geq 0}) \cap U \subset \mathbb P$.
\end{definition}

Clearly, $\mathbb R^n$ is locally open above, and so is every interval given by the intersection of an upset and an open downset.

\begin{theorem}\label{thm:strict_dist}
Let $\mathbb P \subset \mathbb R^n$ be locally open above.
Let $\mathfrak X$ be a connective encoding structure on $\mathbb P$ fulfilling the following two properties:
\begin{enumerate}
    \item Every upset $U \in \mathfrak X$ is closed in $\mathbb{P}$;
    \item Every interval in $\mathbb P$ of the form $[p,q)$, $[p,\infty)$, $(-\infty,q)$, with $p,q \in \mathbb{P}$, is an interval of $\mathfrak X$.
\end{enumerate}
Let $\alpha$ be an amplitude on $\permodx[\mathbb P]$. 
Then $\alpha$ is strict if and only if the induced span metric on $\permodx[\mathbb P]$ induces a metric on isomorphism classes of $\permodx[\mathbb P]$.
\end{theorem}

\begin{remark}
    The assumptions on $\mathfrak X$ in \cref{thm:strict_dist} are fulfilled, for example, for the encoding structure of closed staircases $\cCubes$ on $\mathbb R^n$, as well as for any of the closed below encoding structures in \cref{ex:more_general_classes_of_enc_struct}. 
\end{remark}
For the remainder of the section, we assume $\mathbb P$ and $\mathfrak X$ as in the statement of \cref{thm:strict_dist}.

\begin{remark}
It is possible to drop the strictness assumption on $\alpha$ and locally openness above on $\mathbb{P}$ by passing to the quotient category in which objects with amplitude $0$ are identified with $0$, thus forcefully making $\alpha$ strict.
In many cases, the latter is of the shape $\permodx[\Pos][\mathfrak X']$ for some smaller algebra and $\Pos$ a locally open above subset of $\mathbb R^n$.

For example, if $\mathbb P = [0,1]^n$, $\mathfrak X$ is the encoding structure of cubes and we are considering the Hilbert amplitude with respect to the Lebesgue measure on $\mathbb{R}^n$, then said quotient category is equivalently given by persistence modules encoded by (open above) half open cubes on the half-open cube $[0,1)^n$ (compare \cite{Berkouk_2021}). 
Furthermore, one should note that the following proof of \cref{thm:strict_dist} also works in the case when $\mathbb P$ is discrete, with minor alterations.
\end{remark}

We begin with several remarks on the conditions of \cref{thm:strict_dist} before proving a series of intermediate results, which will result in the proof of \cref{thm:strict_dist}.
First, note that under the connectedness assumptions on $\mathfrak X$, whether two modules are isomorphic can generally be checked at finitely many points. 
This is made formal by the following proposition.

\begin{proposition}\label{prop:comp_set}
Let $\mathfrak X$ be a connective encoding structure over a poset $\Pos$.
Let $M,N$ be objects in $\permodx$.
Then there exists a finite subset of $S \subset \mathbb \Pos$, such that $M$ and $N$ are isomorphic if and only if $i^*_S(M) \cong i^*_S(N)$, where $i_S\colon S \hookrightarrow \Pos$ is the inclusion of posets.
\end{proposition}

\begin{proof}
Choose a common $\mathfrak X$-encoding $e\colon\Pos \to \Pos'$, with $M',N' \in \Vect^{\Pos'}$, of $M$ and $N$.
By \cite[Lem. 3.7]{Lukas}, we may without loss of generality assume that $e$ has $\leq$-connected fibers and that the relations in $\Pos'$ are generated by the images of relations under $e$. 
Denote by $S'$ the set of elements of $\Pos$ given by choosing, for every generating relation in $\Pos'$, a pair of elements $s \leq s' \in \Pos$ mapping to it.
As $\Pos'$ is finite, so is $S'$. 
Furthermore, denote by $S$ the set given by $S'$ together with, for each pair $s\leq s' \in S'$ which are contained in a common fiber of $e$, the elements of a zigzag connecting $s$ and $s'$ in $S'$.  
Then $e'\coloneqq e|_{S}$ fulfills the requirements of \cite[Prop. 3.6]{Lukas} and hence $e'^\ast$ is fully faithful. 
In particular, every isomorphism $e'^\ast M' \cong i_S^*M \xrightarrow{\sim} i_S^\ast N \cong e'^\ast N$ extends to an isomorphism $M' \xrightarrow{\sim} N'$. The latter then pullbacks to an isomorphism of $M$ and $N$ under $e$. 
\end{proof}

We call a subset $S \subset \Pos$ as in \cref{prop:comp_set} a \define{comparison set}.
The following lemma shows that in our framework the property of being a comparison set is stable under small perturbations. 

\begin{lemma}\label{lem:shift_comp_set}
Let $\mathbb P \subset \mathbb R^n$ be locally open above, $\mathfrak X$ an encoding structure given by closed sets on $\mathbb P$, and $S \subset \mathbb P$ a comparison set for $M,N$ in $\ob\permodx[\mathbb P]$. 
Then, for $\varepsilon >0$ small enough, $S + \Vec{\varepsilon}$ is again a comparison set.  
\end{lemma}

\begin{proof}
Using that $S$ is finite, this is immediate from the fact that, as $\mathfrak X$ is given by closed upsets and $\mathbb P$ is locally open for each $s \in S$, the morphisms $M_{s} \to M_{s+ \Vec{\varepsilon}}$ is an isomorphism, for $\varepsilon >0$ small enough. 
\end{proof}
Next, we need the following technical tool.

\begin{lemma}\label{lem:poinwise_iso_lemma}
Let $\mathfrak X$ be any connective encoding structure  on a poset $\Pos$, such that every interval $[p,\infty)$, $(-\infty,p)$, for $p\in \Pos_{-\infty,\infty}$, is an interval of $\mathfrak X$.
Let $\alpha$ an amplitude on $\permodx$ and $M,N$ in $\ob\permodx$, and $I \in \AlgInt$ be such that $M,N$ are constant (i.e. of the form $\Field^k[I]$) when restricted to $I$.
Let $p \in I$ and let \[
\varepsilon \coloneqq \min\{ \alpha(\Field[[p, \infty) \cap I)], \alpha(\Field  [ (-\infty, p] \cap I ] \}.\]  
Consider a span $M \xleftarrow{f} L \xrightarrow{g} N$
with cost $c_{\alpha}(f) + c_{\alpha}(g) < \varepsilon$, and the induced cospan $g'\colon M \to \hat{L} \leftarrow N \colon f'$
obtained by taking the pushout of the former span.
Then both $f'$ and $g'$ are given by isomorphisms at $p$.
\end{lemma}

\begin{proof}
Let us first assume $I = \Pos$.
Then we can consider the following diagram
\[ \begin{tikzcd}
                &   \ker(f) \arrow[two heads, r] \arrow[d] & \ker(f') \arrow[d]   & \\
            \ker(g) \arrow[r] \arrow[two heads, d]   &    L \arrow[d, "f" ] \arrow[r, "g"] & \Field^{k'}\left [I \right ] \arrow[d, "f'"] \arrow[r] \arrow[d]&  \coker(g) \arrow[d, two heads]\\
            \ker(g')    \arrow[r] &   \Field^k\left [I \right ] \arrow[r, "g'"] \arrow[d] & \hat L \arrow[d] \arrow[r] & \coker(g')  \\
            &  \coker(f) \arrow[r, two heads ]& \coker(f') &
\end{tikzcd}\]
with the inner square given by a pushout. As a consequence of the monotonicity of $\alpha$, we obtain that $\alpha(\coker(f)), \alpha(\coker(g)), \alpha(\ker(f')), \alpha(\ker(g')) < \varepsilon$.
Note that, by the epimorphisms in the diagram, for any point that does not lie in the support of the latter four modules, $f'$ and $g'$ are isomorphisms.
Furthermore, all kernels are submodules of a module $\Field^l [I]$, which has all structure morphisms given by monomorphisms.
In particular, any element at $m \in \ker(f)_p$ induces a submodule of the shape $\Field[[p, \infty) \cap I ]$.
Hence, by monotonicity of the amplitude, it follows that $p$ is not contained in the support of the kernels.
An essentially dual argument for the cokernels shows that $p$ is not contained in the support of the cokernels either, and the claim is proved. 

It remains to show that it is always possible to reduce to the case of $\Pos=I$.
To see this, consider the extension by $0$ functor $(-)|^{\Pos}\colon\permodx[I][\mathfrak X|_{I}] \hookrightarrow \permodx$ of \cref{prop:extension_by_zero}.
Since $(-)|^{\Pos}$ is exact, we can transport $\alpha$ to $\permodx[I][\mathfrak X|_{I}]$ along it.
Denote the inclusion $I \hookrightarrow \Pos$ by $i$. Then $i^*$ is an amplitude bounding functor (with constant $1$).
To see this, note that $(i^*M)|^{\Pos}$ is always a subquotient of $M$. Hence, by applying $i^*$ to the setup of the theorem and using that it preserves all finite limits and colimits, we can always reduce to the case $I = \Pos$.
\end{proof}

We now have everything necessary available to finish the proof of \cref{thm:strict_dist}.

\begin{proof}[Proof of \cref{thm:strict_dist}]
The sufficient condition is immediate as $\alpha(M) = \dist_{\alpha}(M,0)$. 

For the converse, assume $M,N$ in $\ob\permodx[\mathbb P]$ with $d_{\alpha}(M,N) = 0$.
Choose a common $\mathfrak X$-encoding map $e\colon \mathbb P \to \Pos$ of $M$ and $N$ and a comparison set $S \subset \mathbb{P}$.
Since $\mathbb{P}$ is locally open above, $S$ is finite, and $\mathfrak X$ is given by closed sets, we can find $\eps >0$ such that each set $s + [0,2\varepsilon)^n = [s,s + \Vec{2\varepsilon}) \subset \mathbb R^n$ is contained in a fiber of $e$, for $s \in S$. 
Thus, $M$ and $N$ are both constant when restricted to any of these intervals. 
In particular, let $\varepsilon$ be small enough, such that (the proof of) \cref{lem:shift_comp_set} applies, and hence $S'\coloneqq S+ \Vec{\varepsilon}$ is still a comparison set for $M$ and $N$.
Next, choose $0<\delta \leq \min\{ \alpha\left(\Field[s, s+\varepsilon), \Field[s+\varepsilon, s + 2\varepsilon)\right) \mid s \in S\}$.
Such a $\delta$ exists, as $\alpha$ is strict and $S$ is finite.
As $\dist_{\alpha}(M,N) = 0$, there exists a span $M \leftarrow L \to N$ of cost $< \delta$.
Then, by \cref{lem:poinwise_iso_lemma}, the pushout-induced cospan $M \to \hat L \leftarrow N$ is given by isomorphisms on $S'$,
i.e. we have an isomorphism $i_{S'}^\ast M  \cong i_{S'}^\ast N$. 
Since $S'$ is a comparison set, the claim is proved.
\end{proof}

\section*{Acknowledgments}
We would like to thank the organizers of the 2020 Applied Category Theory Adjoint School, for bringing us together in this collaboration.
Part of this work was initiated during a postdoctoral fellowship of NO and an internship of LW at the Max Planck Institute for Mathematics in the Sciences, and NO and LW would like to thank the  Max Planck Institute  for their support.
We thank H\aa vard B. Bjerkevik, Wojciech Chach\'olski, Adrian Clough, Nicol\`o De Ponti, Oliver G\"afvert, and Ezra Miller for useful discussions.
BG was supported by the grant FAR2019-UniMORE and by the Austrian Science Fund (FWF) grant numbers P 29984-N35 and P 33765-N. 
NO was partially supported by Royal Society Research Grant RGS\textbackslash R2\textbackslash 212169.
LW was partially supported by Landesgraduiertenförderung Baden-Württemberg.

\bibliographystyle{alpha}

\newcommand{\etalchar}[1]{$^{#1}$}

\end{document}